\documentclass[11pt]{article}

\usepackage{graphicx}
\usepackage{amsmath}
\usepackage{amsthm}
\usepackage{amssymb}
\usepackage[mathscr]{eucal}

\title{Teichm\"{u}ller geometry of moduli space, I: \\
Distance minimizing rays and the Deligne-Mumford compactification}
\author{Benson Farb and Howard Masur \thanks{Both authors are
supported in part by the NSF.}}

\theoremstyle{plain}
\newtheorem{theorem}{Theorem}[section]

\newtheorem{proposition}[theorem]{Proposition}

\newtheorem{lemma}[theorem]{Lemma}
\newtheorem{corollary}[theorem]{Corollary}

\newtheorem{xample}[theorem]{Example}
\newtheorem{definition}[theorem]{Definition}

\def\proof{{\bf {\medskip}{\noindent}Proof. }}

\def\endproof{$\diamond$ \bigskip}

\def\title{\em}

\def\bar{\overline}

\newcommand\R{\mbox{\bf R}}
\newcommand\C{\mbox{\bf C}}
\newcommand\hyp{\mbox{\bf H}}
\newcommand\Z{\mbox{\bf Z}}

\renewcommand\mod{{\rm mod}}

\DeclareMathOperator\Ext{Ext}
\DeclareMathOperator\Area{Area}

\DeclareMathOperator\teich{Teich}
\DeclareMathOperator\Teich{\teich}
\DeclareMathOperator\Mod{Mod}

\DeclareMathOperator\M{{\mathcal M}}

\DeclareMathOperator\B{{\mathcal B}}

\DeclareMathOperator\W{{\mathcal W}}

\DeclareMathOperator\CC{{\mathcal C}}
\DeclareMathOperator\V{{\mathcal V}}
\DeclareMathOperator\MF{{\mathcal MF}}

\DeclareMathOperator\SL{SL}

\DeclareMathOperator\QD{QD}
\DeclareMathOperator\diam{diam}

\DeclareMathOperator\Nbhd{Nbhd}
\DeclareMathOperator\DM{\bar{\M(S)}^{\rm DM}}

\begin{document}
\maketitle


\section{Introduction}

Let $S$ be a {\em surface of finite type}; that is, a closed, oriented surface with a finite (possibly
empty) set of points removed.  In this paper we classify
(globally) geodesic rays in the moduli space $\M(S)$ of Riemann
surfaces, endowed with the Teichm\"{u}ller metric, and we 
determine precisely how pairs of rays asymptote.  We then use
these results to relate two
important but disparate topics in the study of $\M(S)$: 
Teichm\"{u}ller geometry and the
Deligne-Mumford compactification.  We reconstruct the
Deligne-Mumford compactification (as a metric stratified space)
purely from the intrinsic metric geometry of $\M(S)$ endowed with
the Teichm\"{u}ller metric.  We do this by 
constructing an ``iterated EDM ray
space'' functor, which is defined on a quite general class of
metric spaces.  We then prove that this functor applied to
$\M(S)$ produces the Deligne-Mumford compactification.

\medskip
\noindent
{\bf Rays in $\M(S)$.  }A {\em ray} in a metric space $X$ is a map $r:[0,\infty)\to X$ which is
locally an isometric embedding.  In this paper we initiate the study of (globally) isometrically embedded 
rays in $\M(S)$.   Among other things, we classify such rays, determine their asymptotics, classify almost geodesic rays, and work out the Tits angles between rays.  We take as a model for our study the case of rays in locally
symmetric spaces, as in the work of Borel, Ji, MacPherson and others;
see \cite{JM} for a summary.   

In \cite{JM} it is explained how the continuous spectrum of any 
noncompact, complete Riemannian manifold $M$ depends only on the geometry of its
ends, and in some cases (e.g. when $M$ is locally symmetric) 
the generalized eigenspaces can be parametrized by 
a compactification constructed from asymptote classes of certain rays.  
The spectral theory of $\M(S)$ endowed with the Teichm\"{u}ller metric was
initiated by McMullen \cite{Mc}, who proved positivity of the lowest
eigenvalue of the Laplacian.   Our compactification of $\M(S)$ by equivalence classes of certain rays 
might be viewed as a step towards further understanding its spectral theory.   We remark that the 
Teichm\"{u}ller metric is a Finsler metric.

Following \cite{JM}, 
we will consider two natural classes of rays.

\begin{definition}[{\bf EDM rays}]
A ray $r:[0,\infty)\to X$ in a metric space $X$ is {\em eventually
distance minimizing}, or {\em EDM}, 
if there exists $t_0$ such that for all $t\geq t_0$: 
$$d(r(t),r(t_0))=|t-t_0|$$
\end{definition}

Note that, if $r$ is an EDM 
ray, after cutting off an initial segment of $r$ we obtain a globally geodesic ray, i.e. an isometric embedding of $[0,\infty)\to X$ .

\begin{definition}[{\bf ADM rays}]
The ray $r(t)$ is {\em almost distance minimizing}, or {\em ADM}, if 
 there are
constants $C,t_0 \geq 0$ such that for $t\geq t_0$: 
$$d(r(t),r(t_0))\geq
|t-t_0|-C$$ 
\end{definition}

It is easy to check that a ray $r$ is ADM if and only if, for every 
$\epsilon>0$ there exists $t_0\geq 0$ so that for all $t\geq t_0$: 
$$d(r(t),r(t_0))\geq
|t-t_0|-\epsilon$$ 

As with locally symmetric manifolds, there are several ways 
in which a ray in $\M(S)$ might not be ADM: it can traverse a closed geodesic, 
it can be contained in a fixed compact set, or it can return to a fixed
compact set at arbitrarily large times.  More subtly, there are 
rays which leave every compact set in $\M(S)$ and are ADM but are not 
EDM;  these rays ``spiral'' around in 
the ``compact directions'' in the cusp of $\M(S)$. This phenomenon does not appear in 
the classical case of $\M(T^2)=\hyp^2/\SL(2,\Z)$, but it does appear in all moduli spaces of higher complexity, as we shall show.

The set of 
rays in $\Teich(S)$ through a basepoint $Y\in\Teich(S)$ is in bijective correspondence with 
the set of elements $q\in \QD^1(Y)$,
the space of unit area 
holomorphic quadratic differentials $q$ on $Y$ (see \S\ref{section:teichgeom} below).   
We now describe certain kinds of 
Teichm\"{u}ller rays that will be
important in our study.

Recall that a quadratic 
differential $q$ on $Y$ is {\em Strebel} if all of its vertical
trajectories are closed.  In this case $Y$ 
decomposes into a union of flat cylinders. Each cylinder is swept out by 
vertical trajectories of the same length.  
The {\em height} of the cylinder is the distance across the cylinder.

We say $q$ is  {\em mixed Strebel} if it contains at least one 
cylinder of closed trajectories.

\begin{definition}[{\bf (Mixed) Strebel rays}]
A ray in $\M(S)$ is a {\em (mixed) Strebel ray} if it is the projection
to $\M(S)$ of a ray in $\Teich(S)$ corresponding to a pair $(Y,q)$ with $q$
a (mixed) Strebel
differential on $Y$.
\end{definition}

Our first main result is a classification of EDM rays and ADM rays
in moduli space $\M(S)$.

\begin{theorem}[{\bf Classification of {\rm EDM} rays in \boldmath$\M(S)$}]
\label{theorem:rays}
Let $r$ be a ray in $\M(S)$.  Then 
\begin{enumerate}
\item $r$ is {\rm EDM} if and only if it is Strebel.
\item $r$ is {\rm ADM} if and only if it is mixed Strebel.
\end{enumerate}
\end{theorem}

One of the tensions arising from Theorem \ref{theorem:rays} is
that for any $\epsilon>0$, there exist very long local geodesics
$\gamma$ between points $x,y$ in $\M(S)$ which are only
$\epsilon$ longer than any (global) geodesic from $x$ to $y$.  As
distance in $\M(S)$ is difficult to compute precisely, the
question arises as to how such ``fake global geodesics'' $\gamma$
can be distinguished from true global geodesics.  This is done in
\S\ref{section:edm:implies:strebel}.  The idea is to use the
input data of being non-Strebel to build by hand a map whose
$\log$-dilatation equals the length of $\gamma$, but which has
nonconstant pointwise quasiconformal dilatation.  By
Teichm\"{u}ller's uniqueness theorem, since the actual
Teichm\"{u}ller map from $x$ to $y$ has constant pointwise
dilatation, this dilatation, and thus the length of the
Teichm\"{u}ller geodesic connecting $x$ to $y$, is strictly
smaller than the length of $\gamma$.

We also determine finer information about EDM rays.  In Section
\ref{section:asymptote} we determine the limiting asymptotic
distance between EDM rays: it equals the Teichm\"{u}ller distance
of their endpoints in the ``boundary moduli space'' (see Theorem
\ref{thm:distance} below).  This precise behavior of rays in
$\M(S)$ lies in contrast to the behavior of rays in the
Teichm\"{u}ller space of $S$, which themselves may not even have
limits.  Theorem \ref{thm:distance} is crucial for our
reconstruction of the Deligne-Mumford compactification.  In
Section \ref{section:tits} we compute the Tits angle of any two
rays, showing that only $3$ possible values can occur.  This result 
contrasts with the behavior in locally symmetric manifolds, where 
a continuous spectrum of Tits angles can occur.

\bigskip \noindent 
{\bf Reconstructing the topology of
Deligne-Mumford.  }Deligne-Mumford \cite{DM} constructed a
compactification $\DM$ of $\M(S)$ whose points are represented by
conformal structures on noded Riemann surfaces. They proved that
$\DM$ is a projective variety.  As such, $\DM$ as a topological
space comes with a natural stratification: each stratum is a
product of moduli spaces of surfaces of lower complexity.  We
will equip each moduli space with the Teichm\"{u}ller metric, and
the product of moduli spaces with the $\sup$ metric.  In this way
$\DM$ has the structure of a {\em metric stratified space},
i.e. a stratified space with a metric on each stratum (see
\S\ref{section:irdm} below).  We note that $\DM$ was also
constructed topologically by Bers in \cite{Be}.

In Section \ref{section:irdm} we construct, for any geodesic
metric space $X$, a space $\bar{X}^{ir}$ of $X$,
called the {\em iterated EDM ray space} associated to $X$.  This 
space comes from considering asymptote classes of EDM
rays, endowing the set of these with a natural metric, and then
considering asymptote classes of EDM rays on this space, etc.
The space $\bar{X}^{ir}$ has the structure of a metric stratified
space.

\begin{theorem} Let $S$ be a surface of finite type.  Then there
is a strata-preserving homeomorphism $\bar{\M(S)}^{ir}\to \DM$
which is an isometry on each stratum.  \end{theorem}

Thus, as a metric stratified space, $\DM$ is determined by the
intrinsic geometry of $\M(S)$ endowed with the Teichm\"{u}ller
metric.  The following table summarizes a kind of dictionary
between purely (Teichm\"{u}ller) metric properties of $\M(S)$ on
the one hand, and purely combinatorial/analytic properties on the
other.  Each of the entries in the table is proved in this paper.

\bigskip
\begin{center}
\begin{tabular}{|l|l|}
\hline
PURELY METRIC & ANALYTIC/COMBINATORIAL\\
\hline
EDM ray in $\M(S)$& Strebel differential \\
\hline
ADM ray in $\M(S)$& mixed Strebel differential\\
\hline
isolated EDM ray in $\M(S)$& one-cylinder Strebel differential\\
\hline
asymptotic EDM rays in $\M(S)$& modularly equivalent Strebel differentials \\
&with same endpoint\\
\hline
iterated EDM ray space of $\M(S)$& Deligne-Mumford compactification $\DM$\\
\hline
rays of rays of $\cdots$ of rays ($k$ times)& level $k$ stratum of $\DM$\\
\hline
Tits angle 0 & pairs of combinatorially equivalent \\
&Strebel differentials\\
\hline
Tits angle 1 & pairs of Strebel differentials with \\
&disjoint cylinders\\
\hline
Tits angle 2 & all other pairs of Strebel differentials\\
\hline
\end{tabular}
\end{center}

\bigskip
\noindent
{\bf Acknowledgements. }We would like to thank Steve Kerckhoff, Cliff Earle, and Al Marden and Yair Minsky for useful discussions, and Chris Judge for numerous useful comments and corrections.  We are also grateful to Kasra Rafi for his crucial help relating to the appendix.

\section{Teichm\"{u}ller geometry and extremal length}
\label{section:teichgeom}

In this section we quickly explain some basics of the Teichm\"{u}ller metric and quadratic differentials.  We also make some extremal length estimates which will be used later.  
The notation fixed here will be used throughout the paper.

Throughout this paper $S$ will denote a surface of {\em finite type}, by which we mean a closed, oriented surface with a (possibly empty)
finite set of points deleted.  We call such deleted points {\em
punctures}.  The {\em Teichm\"{u}ller space} $\Teich(S)$ is the
space of equivalence classes of marked conformal structures
$(f,X)$ on $S$, where two markings $f_i:S\to X_i$ are equivalent
if there is a conformal map $h:X_1\to X_2$ with $f_2$ homotopic to $h\circ f_1$.  We often 
drop the marking notation, remembering that a marked surface is the same as a surface where we ``know the names of the curves''.

The {\em Teichm\"{u}ller metric} on $\Teich(S)$ is the metric
defined by $$d_{\Teich(S)}((X,g), (Y,h)):=\frac{1}{2}\inf\{\log
K(f): f:X\to Y \mbox{\ is homotopic to\ }h\circ g^{-1}\}$$ where
$f$ is quasiconformal and $$ K(f):={\rm ess}-\sup_{x\in S}K_x(f)\geq 1 $$
\noindent is the {\em quasiconfromal dilatation} of $f$, where
$$K_x(f):=
\frac{|f_z(x)|+|f_{\bar{z}}(x)|}{|f_z(x)|-|f_{\bar{z}}(x)|} $$
\noindent is the {\em pointwise quasiconformal dilatation} at
$x$.  We also use the notation $d_{\Teich(S)}(X,Y)$ with the
markings implied.  The {\em mapping class group} $\Mod(S)$ is the
group of homotopy classes of orientation-preserving
homeomorphisms of $S$.  This group acts properly discontinuously
and isometrically on $(\Teich(S),d_{\Teich(S)})$, and so the
quotient $$\M(S)=\Teich(S)/\Mod(S)$$ has the induced metric.  $\M(S)$ is the moduli space of
(unmarked) Riemann surfaces, or what is the same thing, conformal
structures on $S$.

\subsection{Quadratic differentials and Teichm\"{u}ller rays}

\noindent
{\bf Quadratic differentials and measured foliations. }  Let $S$ be a surface of finite type, and 
let $X\in\Teich(S)$.  Recall that a {\em (holomorphic) quadratic differential} $q$ on $X$ is a tensor 
given in holomorphic local coordinates $z$ by $q(z)dz^2$, where  $q(z)$ is holomorphic.  Let $\QD(X)$ denote the space of holomorphic quadratic differentials on $X$.   Any $q\in\QD(X)$ determines a singular Euclidean metric $|q(z)||dz|^2$, with the finitely many singular points corresponding to the zeroes of $q$.  The total area of $X$ in this metric is finite, and is denoted by $\|q\|$, which is a norm on $\QD(X)$.  We denote by $\QD^1(X)$ the set of elements $q\in\QD(X)$ with
$\|q\|=1$.  

An element $q\in\QD(X)$ determines a pair of transverse measured foliations 
${\mathcal F}_h(q)$ and ${\mathcal F}_v(q)$, called the {\em horizontal and vertical foliations} for $q$.  The leaves of these foliations are paths $z=\gamma(t)$ such that $$q(\gamma(t))\gamma'(t)^2>0$$ and $$q(\gamma(t))\gamma'(t)^2<0,$$ 
In a neighborhood of a nonsingular point, there are {\em natural} coordinates $z=x+iy$ so that 
the leaves of ${\mathcal F}_h$ are 
given by $y={\rm const.}$, the leaves of ${\mathcal F}_v$ are given by $x={\rm const.}$, and the transverse measures are  $|dy|$ and $|dx|$.  The foliations ${\mathcal F}_h$ and ${\mathcal F}_v$ have the zero set of $q$ as their common singular set, and at each zero of order $k$ they have a $(k+2)$-pronged singularity, locally modelled on the singularity at the origin of $z^kdz^2$.  
The leaves passing through a singularity are the {\em singular leaves} of the measured foliation.  A {\em saddle connection} is a leaf joining two (not necessarily distinct) singular points.  
The union of the saddle connections of the vertical foliation is called the {\em critical graph} $\Gamma(q)$ of $q$.   

The components $X\setminus \Gamma(q)$ are of two types: cylinders swept out by vertical trajectories (i.e. leaves of ${\mathcal F}_v$) of equal length, and {\em minimal components} where each leaf of ${\mathcal F}_v$ is dense.

\medskip
\noindent  
{\bf Teichm\"{u}ller maps and rays.  } 
{\em Teichm\"{u}ller's Theorem}  states that, given any $X,Y\in\Teich(S)$, there exists a unique 
(up to translation in the case when $S$ is a torus) 
quasiconformal map $f$, called the {\em Teichm\"{u}ller map}, realizing $d_{\Teich(S)}(X,Y)$.   
The Beltrami coefficient $\mu:=\frac{\bar{\partial }f}{\partial f}$ is of the form $\mu=k\frac{\bar{q}}{|q|}$ for 
some $q\in\QD^1(X)$ and some $k$ with $0\leq k<1$.  In natural 
local coordinates given by $q$ and a quadratic differential $q'$ on $Y$, we have 
$f(x+iy)=Kx+\frac{1}{K}iy$, where $K=K(f)=\frac{1+k}{1-k}$.  Thus $f$ dilates the horizontal foliation by $K$ and the vertical foliation by $1/K$.   

Any $q\in\QD^1(X)$ determines a geodesic ray $r=r_{(X,q)}$ in $\Teich(S)$, called the 
{\em Teichm\"{u}ller ray} based at $X$ in the direction of $q$.   The ray $r$ is given by the complex structures determined by the quadratic differentials $q(t)$ obtained by 
multiplying the transverse measures of ${\mathcal F}_h(q)$ and ${\mathcal F}_v(q)$ by $\frac{1}{K}=e^{-t}$ and $K=e^t$, respectively, for $t>0$.  To summarize,  for each $X\in\Teich(S)$, 
there is a bijective correspondence between the set of rays in $\Teich(S)$
based at $X$ and the set of elements of $\QD^1(X)$.  

Finally, we note that any ray in ${\mathcal M}(S)$ is 
the image of a ray in $\Teich(S)$ under the natural quotient map 
$$\Teich(S)\to {\mathcal M}(S)=\Teich(S)/\Mod(S).$$

\subsection{Extremal length and Kerckhoff's formula}

Kerckhoff \cite{Ke} discovered
an elegant and useful way to compute Teichm\"{u}ller distance in terms
of extremal length, which is a conformal invariant of isotopy classes of
simple closed curves.  We now describe this, following \cite{Ke}.

Recall that a {\em conformal metric} on a Riemann surface $X$ is a
metric which is locally of the form $\rho(z)|dz|$, where $\rho$ is a
non-negative, measurable, real-valued function on $X$.  A conformal metric
determines a length function $\ell_\rho$, which assigns to each (isotopy
class of) simple closed curve $\gamma$ the infimum $\ell_\rho(\gamma)$
of the lengths of all curves in the isotopy class, where length is
measured with respect to the conformal metric.   We denote the area of
$X$ in a conformal metric given by a function $\rho$ by $\Area_\rho(X)$,
or $\Area_\rho$ when $X$ is understood.

By {\em cylinder} \ we will mean the surface $S^1\times
[0,1]$, endowed with a conformal metric.  Recall that any cylinder $C$ is
conformally equivalent to a unique annulus 
of the form $\{z\in\C: 1\leq |z|\leq r\}$.  The number $(\log r)/2\pi$
will be called the {\em modulus} of $C$, denoted $\mod(C)$.  
A {\em cylinder in $X$} is an
embedded cyclinder $C$ in $X$, endowed with the conformal 
metric induced from the conformal metric on
$X$. 
There are two
equivalent definitions of extremal length, each of which is useful.

\begin{definition}[Extremal length]
Let $X$ be a fixed Riemann surface, and let $\gamma$ be an isotopy
class of simple closed curves on $X$.  The {\em extremal length} of
$\gamma$ in $X$, denoted by $\Ext_X(\gamma)$, or $\Ext(\gamma)$ when $X$
is understood, is defined to be one of the following 
two equivalent quantities:

\begin{description}
\item[Analytic definition: ]
$$\Ext(\gamma):=\sup_\rho \ell_\rho(\gamma)^2/\Area_\rho$$
where the supremum is over all conformal metrics $\rho$ on $X$ 
of finite positive area.

\item[Geometric definition: ]$$\Ext(\gamma):=\inf \{\frac{\displaystyle
1}{\displaystyle\mod(C)}: \mbox{$C$
is a cylinder with core curve isotopic to $\gamma$}\}$$
\end{description}

\end{definition}

As pointed out by Kerckhoff in \cite{Ke}, and as we will see throughout
the present paper, the analytic definition 
is useful for finding lower bounds for $\Ext(\gamma)$, while the
geometric definition is useful for finding upper bounds.  

\begin{theorem}[Kerckhoff \cite{Ke}, Theorem 4]
\label{theorem:kerckhoff:formula}
Let $S$ be any surface of finite type, and let 
$X,Y$ be any two points of $\Teich(S)$.  Then 
\begin{equation}
\label{eq:kerckhoff:formula}
d_{\Teich(S)}(X,Y)=\frac{1}{2}\log \ [\ \sup_\gamma 
\frac{\displaystyle \Ext_X(\gamma)}{\displaystyle \Ext_Y(\gamma)}\ ]
\end{equation}
where the supremum is taken over all isotopy classes of simple closed
curves $\gamma$ on $S$.
\end{theorem}

\begin{remark} The definition of extremal length is easily
extended to measured foliations.  The density of simple closed
curves in the space $\MF(S)$ of measured foliations on $S$ 
allows us to replace the right
hand side of (\ref{eq:kerckhoff:formula}) by the supremum taken
over all $\gamma\in\MF(S)$. \end{remark}

\subsection{Extremal length estimates along Strebel rays}

Let $(X,q)$ be a Riemann surface $X\in\Teich(S)$ with Strebel
differential $q\in\QD(X)$, and let $r=r_{(X,q)}$ be the
corresponding Strebel ray.  Our goal in this subsection is to
estimate the extremal length $\Ext_{r(t)}(\beta)$ of an arbitrary
(isotopy class of) simple closed curve $\beta$ as the underlying
Riemann surface moves along the ray $r$.  The following estimates
are due to Kerckhoff \cite{Ke}.  We include proofs here for
completeness, and because these estimates are so essential for
this paper.

The setup will be as follows.  Let $C_i, 1\leq i\leq n$ be 
the cylinders of the Strebel differential $q$, and for each $i$ let
$\alpha_i$ denote the homotopy class of the core curve of $C_i$.  
Let $a_i(t)$ denote the
$q(t)$-length of $\alpha_i$ and let $b_i(t)$ denote the 
$q(t)$-height of $C_i$.   
Let $M_i(t)=\text{mod}(C_i)=b_i(t)/a_i(t)$ be 
the modulus. Note that on the Riemann surface $r(t)$ we have 
$$a_i(t)=e^{-t}a_i(0)$$
and
the height $b_i(t)$ of the cylinder $C_i$ satisfies $$b_i(t)=e^tb_i(0).$$  

Recall that the {\em geometric intersection number} of two isotopy classes of simple closed curves $\alpha,\beta$, denoted $i(\alpha,\beta)$, is the miminal number of intersection points of curves 
$\alpha'$ and $\beta'$ isotopic to $\alpha$ and $\beta$, respectively.

\begin{lemma} 
\label{lem:short1}
With notation as above, the following hold:
\begin{enumerate}
\item 
$\lim_{t\to\infty}e^{2t}M_i(0)\text{Ext}_{r(t)}(\alpha_i)=1.$
\item There is a constant $c>0$  such that if $i(\beta,\alpha_i)=0$ for all $i$ and $\beta$  is not isotopic to any of the $\alpha_i$, then for all $t$ large enough,  $$\text{Ext}_{r(t)}(\beta)\geq c.$$
\item There is a constant $c>0$ such that if  $\beta$ crosses $C_i$  
 then for $t$ large enough, $$\text{Ext}_{r(t)} (\beta)\geq ce^{2t}.$$

\end{enumerate}

\end{lemma}
\begin{proof}
To prove Statement (1) we recall that the geometric definition of extremal length says  that $$\Ext_{r(t)}(\alpha_i)=
\inf\frac{\displaystyle1}{\displaystyle \mod(A)},$$ where the infimum 
is taken over all cylinders $A\subset r(t)$ homotopic to
$\alpha_i$. 
Statement 1 is immediate in the case that $n=1$, for then by Theorem 20.4(3 ) of \cite{St}, taken with $i=1$,  the modulus of a 
one-cylinder Strebel differential realizes the supremum of the moduli of all cylinders homotopic to $\alpha_1$, so that the reciprocal realizes the infimum of the reciprocals of the moduli in the geometric definition. In that case 
the limit in Statement 1 is actually an equality  for each $t$.   Thus assume $m>1$. 
 On $r(t)$, the cylinder $C_i$ has modulus $e^{2t}b_i(0)/a_i(0)=e^{2t}M_i(0)$, giving the bound
$$\text{Ext}_{r(t)}(\alpha_i)\leq \frac{e^{-2t}}{M_i(0)}.$$

We now give a lower bound. We can realize the surface $r(t)$ by
cutting along the core curves of the cylinders, that are halfway across eachcylinder,  inserting
cylinders of circumference $a_i(0)$ and height $\frac{b_i(0)(e^{2t}-1)}{2}$ to each side of the cut and then regluing. Rescaling by $e^t$ the flat metric
induced by $q(t)$, gives a flat metric $\rho(t)$ of area $e^{2t}$
for which the core curves have constant length $a_i(0)$ and
height $e^{2t}b_i(0)$.  Choose a constant $b$ such that
$b>a_i(0)$, and for $t_0$ sufficiently large, choose a fixed
neighborhood $\Nbhd(C_i)$ of $C_i$ on $r(t_0)$ such that
$$d_{\rho(t_0)}(C_i,\partial \Nbhd(C_i))=b.$$ For some fixed $B>0$
we have $$\text{area}_{\rho(t_0)}(\Nbhd(C_i)\setminus C_i)=B.$$ Via the construction described above, we
may
 think of $\Nbhd(C_i)\setminus C_i$ as a subset of $r(t)$ for $t\geq t_0$.  
 Define a conformal metric $\sigma_i(t)$ on  $r(t)$ as follows. 
It is  given by  $\rho(t)$ on $C_i$.  On $\Nbhd(C_i)\setminus
C_i$ it is given by the metric  $\rho(t_0)$, and 
on  $r(t)\setminus \Nbhd(C_i)$  it is 
given by $\delta \rho(t)$ for some $\delta>0$.      
With respect to the metric $\sigma_i(t)$ we then have  
$$d_{\sigma_i(t)}(C_i,\partial \Nbhd(C_i))=b$$ and

$$\text{Area}_{\sigma_i(t)}\leq B+\delta
e^{2t}+e^{2t}a_i(0)b_i(0).$$ Since the distance across
$\Nbhd(C_i)\setminus C_i$ is at least $b\geq a_i(0)$, it is easy
to see that $$\ell_{\sigma_i(t)}(\alpha_i)=a_i(0).$$ Putting the
estimates on lengths and areas together, it follows that given
any $\epsilon>0$, we may choose $\delta>0$ so that for $t$ large
enough,
$$\text{Ext}_{r(t)}(\alpha_i)\geq\frac{\ell_{\sigma_i(t)}^2(\alpha_i)}{A_{\sigma_i(t)}}\geq
(1-\epsilon)\frac{e^{-2t}}{M_i(0)} .$$ Putting this lower bound
together with the upper bound we have proved (1).

For the proof of (2), for $t_0$ large enough, take a fixed neighborhood $N$ of the
component of the critical graph $\Gamma$ that contains $\beta$
such that the distance across $N$ is at least $\min_i a_i(0)$, the 
lengths of the core curves of the cylinders on the base surface $r(0)$.  Again we may consider $N$ as
a subset of $r(t)$ for all $t\geq 0$.  We put a conformal metric
$\sigma(t)$ on $r(t)$ which is given by the flat metric defined by $q(t)$ on
$r(t)\setminus N$ and the metric defined by $q(t)$ scaled by $e^t$ on $N$.  For some fixed $B>0$ we
have $$\text{Area}_{\sigma(t)}\leq B.$$ Now any geodesic representative of $\beta$ that enters $r(t)\setminus N$ must bound  a disc with a core curve of $C_i$, and can be shortened to lie entirely inside $N$.  Thus its geodesic representative in fact lies in the critical graph and so there is a $b$ such that
$$\ell_{\sigma(t)}(\beta)\geq b.$$ The lower bound now follows
from these last two inequalities and the analytic definition of extremal length.

The proof of (3) follows by using the given metric $q(t)$ in the analytic definition of extremal length.  
\end{proof}

\section{EDM and ADM rays in moduli space}

In this section we classify EDM and ADM rays in moduli space,
giving a proof of Theorem \ref{theorem:rays}.  We then determine,
in \S\ref{section:asymptote}, when two EDM rays are asymptotic.   

\subsection{Strebel rays are EDM}

Our goal in this subsection is to prove one direction of 
Theorem \ref{theorem:rays}, namely that 
if $(X,q)$ is Strebel then the ray $r_{(X,q)}$ in $\M(S)$ 
is eventually distance minimizing.   

Since $\Mod(S)$ acts properly by isometries on $\Teich(S)$ with quotient $\M(S)$, the distance between points $x,y\in \M(S)$ are the same as minimal distances between orbits of any lift of $x,y$ to $\Teich(S)$.   We warn the reader that while 
every ray in $\M(S)$ comes from the projection to $\M(S)$ of a ray in $\Teich(S)$, the converse is not true; this is due to the fixed points of the action of $\Mod(S)$ on $\Teich(S)$.  

Thus, to achieve our goal, we must find $t_0\geq 0$ so 
that
\begin{equation}
\label{eq:edmteich} 
d_{\Teich(S)}(r(t),r(t_0))\leq d_{\Teich(S)}(\phi(r(t_0)),r(t))
\end{equation}
for all $t\geq t_0$ and for every $\phi\in\Mod(S)$.  
In fact we will prove for Strebel rays that the inequality in (\ref{eq:edmteich}) is
strict for $t>t_0$, as long as $\phi$ doesn't have a fixed point.

\begin{remark}
Note that while any two nonseparating curves
on $S$ can be taken to each other via some element of $\Mod(S)$, Strebel
rays along cylinders with nonseparating core curves, based at the same
$Y\in\Teich(S)$, project to different rays in $\M(S)$.  Indeed, given 
any point $X\in\M(S)$, there are countably infinitely many Strebel
rays in $\M(S)$ based at $X$, even though there are $[g/2]+1$
topological types of simple closed curves on $S$.  
\end{remark}

\bigskip

Let $\alpha_1,\ldots ,\alpha_p$ denote the core curves of the cylinders $\{C_i\}$ in the cylinder decomposition of $(X,q)$.  By Lemma \ref{lem:short1}, the extremal length of
curves $\beta$ with $i(\beta,\alpha_i)=0$ for each $1\leq i\leq p$ and not homotopic to any $\alpha_i$ remain bounded below by some $d>0$. By Lemma~\ref{lem:short1} 
the extremal length of any 
curve $\beta$ with $i(\beta,\alpha_i)>0$ for some $i$ tends to $\infty$ as $t\to\infty$.
Choose $t_0$ big enough so that each of the following holds:
\begin{enumerate}
\item If $i(\beta,\alpha_i)>0$ for some $i$, then $\Ext_{r(t)}(\beta)\geq d$ for $t\geq t_0$.
\item  $e^{2t_0}>2\max_i(\frac{M_i}{d})$, where $M_i$ is the modulus of the cylinder $C_i$   
 \item For $t\geq t_0$, $\text{Ext}_{r(t)}(\alpha_i)\leq
2e^{-2t}M_i$.  (This can be done by 
 Lemma~\ref{lem:short1}). 
 \end{enumerate}
 
Let $\phi$ be any element of $\Mod(S)$ without a fixed point in $\Teich(S)$; this is the same as $\phi$ not having finite order.  Suppose first that 
$\phi^{-1}(\alpha_i)=\beta\notin\{\alpha_j\}$ for some $i$.  By Theorem 
\ref{theorem:kerckhoff:formula} we have for $t>t_0$:
$$
\begin{array}{ll}
d_{\Teich(S)}(\phi(r(t_0)),r(t))&\geq  \frac{1}{2}\log
\frac{\displaystyle \Ext_{\phi(r(t_0))}(\alpha_i)}{\displaystyle 
\Ext_{r(t)}(\alpha_i)}\\
&\\
&=\frac{1}{2}\log\frac{\displaystyle\Ext_{r(t_0)}(\beta)}{\displaystyle
\Ext_{r(t)}(\alpha_i)}\\
&\\
&\geq\frac{1}{2}\log \frac{\displaystyle d}{\displaystyle 2e^{-2t}M_i}\\
&\\
&>\frac{1}{2}\log \frac{\displaystyle e^{2t}}{\displaystyle e^{2t_0}}=t-t_0=d_{\Teich(S)}(r(t),r(t_0))
\end{array}.$$

Thus we may assume that $\phi$ preserves $\{\alpha_i\}$ as a set.  Consider the special 
case when $\phi(\alpha_i)=\alpha_i$ for each $i$. This assumption
implies that $\phi^{-1}$ preserves the vertical foliation ${\mathcal F}_v(q)$ of $q$, as a
measured foliation.  Then 
\begin{equation}
\label{eq:edm:streb}
d_{\Teich(S)}(\phi(r(t_0)),r(t))\geq\frac{1}{2} \log
\frac{\Ext_{\phi(r(t_0))}({ \mathcal F}_v(q))}{\Ext_{r(t)}({\mathcal F}_v(q))}= \frac{1}{2}\log
\frac{\Ext_{r(t_0)}({\mathcal F}_v(q))}{\Ext_{r(t)}({\mathcal F}_v(q))}=t-t_0
\end{equation}
and we are again done in this case.  The leftmost inequality follows from the remark after 
Theorem \ref{theorem:kerckhoff:formula}.

We remark that the inequality (\ref{eq:edm:streb}) is strict.  This is because equality
of the leftmost terms occurs if and only if ${\mathcal F}_v(q)$ is 
the vertical foliation of the
quadratic differential defining the Teichmuller map from $\phi(r(t_0))$ to
$r(t)$.  However, ${\mathcal F}_v(q)$  is the vertical foliation of the quadratic
differential of the Teichmuller map from $r(t_0)$ to $r(t)$, and so
it cannot be the former since $\phi$ is assumed to be nontrivial.

Finally, consider the general case of $\phi$ preserving
$\{\alpha_i\}$ as a set.  Let $k$ be the smallest integer such that
$\phi^k(\alpha_i)=\alpha_i$ for all $i$.  If the desired result is not
true there is a sequence of times $t_0<t_1<\ldots <t_k$ such that
$$d_{\Teich(S)}(r(t_{i-1}),\phi(r(t_i)))<
d_{\Teich(S)}(r(t_{i-1}),r(t_i)).$$ 

Since $\phi$ acts as an isometry of $\Teich(S)$, applications of the
triangle inequality give

$$d_{\Teich(S)}(r(t_0),\phi^k(r(t_k)))<d_{\Teich(S)}(r(t_0),r(t_k)).$$ 

But $\phi^k$
fixes each $\alpha_i$, and we have a contradiction to the previous
assertion.

\subsection{Every EDM ray is Strebel}
\label{section:edm:implies:strebel}

In this subsection we prove the other direction of
Theorem \ref{theorem:rays}, namely that if a ray
$r_{(X,q)}:[0,\infty)\to\M(S)$ is EDM then $(X,q)$ is Strebel.
The idea of the proof is explained in the introduction above.  Since $r$ is EDM, we can change 
basepoint and assume that $r$ is (globally) isometrically embedded.  We henceforth assume this.

Recall that for each $t\geq 0$, the ray $r=r_{(X,q)}$ determines the
following data: the Riemann surface
$r(t)\in\M(S)$, the quadratic differential $q(t)\in\QD^1(r(t))$,
and the vertical
foliation ${\mathcal F}_v(q(t))$ for the quadratic differential $q(t)$.  
Let $\Gamma(t)$ denote the critical graph of $q(t)$, so that $\Gamma(t)$
is the union of the vertical saddle connections of $q(t)$.  Note that
$\Gamma(t)$ may be empty.  

For any quadratic differential $q$ on a Riemann surface $X$, 
let $\Sigma$ denote the set of zeroes of $q$.
We define the {\em diameter} of $X$ (in the $q$-metric $d_q$), denoted $\diam(X)$, to be 
$$\diam(X):=\sup_{x\in X} d_q(x,\Sigma).$$

Now suppose that the ray $r=r_{(X,q)}$ is not Strebel.  This assumption
implies that there is some subsurface $Y(t)\subseteq r(t)$ 
which contains some leaf of ${\mathcal F}_v(q(t))$ which is dense in $Y(t)$. 
We will find a contradiction. 

\bigskip
\noindent
{\bf Step 1 (Delaunay triangulations): }  

\begin{proposition} 
\label{proposition:delaunay} 
There is a
triangulation $\Delta(t)$ on $r(t)$ with the following
properties: 
\begin{enumerate} 
\item The vertices of $\Delta(t)$
lie in the zero set of $q(t)$.  
\item The edges of $\Delta(t)$
are saddle connections of $q(t)$.  
\item For $t$ large enough,
every edge of the vertical critical graph $\Gamma(t)$ is an edge
of $\Delta(t)$.  
\item There is a function $c(t)$ with $c(t)\to
\infty$ as $t\to \infty$ so that every triangle in $\Delta(t)$
whose interior is contained in some minimal component $Y$, can be
inscribed in a circle of radius at most $e^t/c(t)$.
 \end{enumerate}
\end{proposition}

\begin{proof} The triangulation $\Delta(t)$ will be the {\em
Delaunay triangulation} $\Delta(t)$ constructed by Masur-Smillie
in \S 4 of \cite{MS}.  In particular, $\Delta(t)$ automatically satisfies
(1) and (2).  We now claim something very special about $\Delta(t)$.

\begin{lemma}
\label{lem:short2}
There is a function $c(t)$ with $\lim_{t\to\infty} c(t)=\infty$ with the
following property: the shortest saddle connection $\beta(t)$ of the quadratic
differential $q(t)$ on $r(t)$, whose endpoints lie in $\bar{Y}\cap \Sigma$, and whose interior lies in $Y$, 
has length 
at least $c(t)e^{-t}$.
\end{lemma}

\begin{proof}[of Lemma \ref{lem:short2}]
Denote by  $|\cdot|_t$ the length function associated to flat metric on $r(t)$ induced by $q(t)$.   
For an arc $\alpha$, we denote by $|\alpha|^{\rm vert}_t$ (resp. $|\alpha|^{\rm horiz}_t)$ the length of 
$\alpha$ as measured with respect to the transverse measure $|dy|$ on ${\mathcal F}_h(q(t))$ 
(resp. $|dx|$ on ${\mathcal F}_v(q(t))$).  

We claim that there is a constant $D$ such that $|\beta(t)|_t\leq
D$.  To prove the claim, consider an edge $E$ of the Delaunay
triangulation $\Delta(t)$ with $E\cap Y\neq \emptyset$.  First suppose
$|E|_t\leq s=:2\sqrt{2/\pi}$.  If $E\subset Y$ then take $D=s$
and we are done.  If $E$ is not contained in $Y$, then it crosses
some edge $\alpha$ of $\Gamma(t)$.  We remind the reader that, as
we move out along $r(t)$, the horizontal lengths are expanded by
$e^t$ and the vertical length are contracted by $e^{-t}$.  Thus
we have the equation 

\begin{equation} \label{eq:crit2}
|\alpha|_t=e^{-t}|\alpha|_{0}.  
\end{equation}

But then we can take some subsegment of $E$, together with a
union of at most two subsegments of $\Gamma(t)$, to give a
nontrivial homotopy class of arc with endpoints in $\bar{Y}\cap
\Sigma$ and interior contained in $Y$.  The geodesic
representative $\beta(t)$ in this homotopy class has length
bounded above by the length of $E$ plus the length of
$\Gamma(t)$, which for large enough $t$ is less than $D=s+1$, and
we are done.

We are now reduced to the case where $E$ has length at least $s$.  By Proposition 5.4 of \cite{MS}, $E$ must cross some flat cyclinder 
$C$ in $r(t)$ whose height is greater than its circumference.  If $C\subset Y$, then since 
$Y$ has area at most $1$, the circumference is at most $1$, and so taking $\beta(t)$ to be the circumference, we have $|\beta(t)|_t\leq 1$.  If $C$ is not contained in $Y$, then $C$ crosses the critical graph $\Gamma(t)$.   Thus the height of $C$ is bounded, as in  (\ref{eq:crit2}).  Thus the 
circumference is bounded as well.  An argument similar to the previous paragraph then provides 
$\beta(t)$, and the claim is proved.

We now continue with the proof of the lemma.  We have 
$$|\beta(t)|_t\geq  |\beta(t)|^{\rm
horiz}_t=e^{t}|\beta(t)|^{\rm horiz}_{0}.$$
Since $|\beta(t)|_t$ is bounded, we must have  
 $|\beta(t)|^{\rm horiz}_{0}\to 0$ as $t\to\infty$. Since $Y$ is assumed to be minimal, there are no vertical saddle connections in $Y$,  and so 
$|\beta(t)|^{\rm horiz}_{0}>0$.  Because the set of holonomy vectors of saddle connections is a 
discrete subset of $\R^2$ for a fixed flat structure  (see, e.g, \cite{HS}), this forces 
$|\beta(t)|_{0}^{\rm vert}\to \infty$ as $t\to\infty$. 
Now $$|\beta(t)|_t \geq |\beta(t)|_t^{\rm vert}=e^{-t}|\beta(t)|_0^{\rm vert}.$$
Thus the desired inequality holds with $c(t)=|\beta(t)|^{\rm vert}_{0}$.
\end{proof}

We continue with the proof of Proposition \ref{proposition:delaunay}.

For 
$t$ sufficiently large, by Lemma~\ref{lem:short2} the segments of $\Gamma(t)$ 
are the shortest saddle connections on $q(t)$.   Now since these segments are all
vertical, given such a segment $\alpha$, the midpoint $p$ of $\alpha$ has the property that the two endpoints of $\alpha$ realize the distance from $p$ to $\Sigma$.  Thus, by construction of 
the Delaunay triangulation (see \S 4 of \cite{MS}), the entire segment $\alpha$ lies in $\Delta(t)$.  This proves (3).

We now prove (4).  By (3), no edge of the triangulation crosses $\Gamma(t)$, so any $2$-cell that intersects a minimal component $Y$ is contained in that minimal component.   By Theorem 4.4 of \cite{MS}, every point in $Y$ is contained in a unique Delaunay cell isometric to a polygon inscribed in a circle of radius $\leq \diam(Y)$.  It therefore remains to bound $\diam(Y)$.  

If $\diam(Y)>2s$, then there is a cylinder $C$ whose height is at least $s$.  As we have seen above, such a cylinder must be contained in $Y$, as it cannot cross $\Gamma(t)$ for sufficiently large $t$.  
But Lemma \ref{lem:short2} gives that the circumference of $C$ is at least $c(t)e^{-t}$, and since $r(t)$ has unit area, the height of $C$ is at most $e^t/c(t)$, and there the diameter is 
at most $(e^t/2c(t))+1$ (the second term coming from a bound on the length of the 
circumference of $C$).
\end{proof}

We now return to the proof that EDM rays are Strebel.  Recall we are arguing by contradiction, 
so that we are assuming the quadratic differential $q_0$ defining the say is not Strebel. 
 Thus there is at least one minimal component in the complement of the critical graph of $q_0$. Let $C_1,\ldots ,C_r$ be the (possibly empty)
collection of vertical cylinders of $q_0$.  

Recall now that for $t$ large enough, by Proposition~\ref{proposition:delaunay}, the critical graph of $q(t)$ are edges of the Delaunay triangulation $\Delta(t)$ of $q(t)$.  Consider $\Delta(t)$ restricted to the complement of the cylinders $C_i$.  

\begin{proposition}
\label{proposition:finite:triang}
Let $(r_0, q_0)$ be given with (possibly empty) cylinder data.  There
exist finitely many triangulations $T_1,\ldots,T_m$ of the complement of the set of cylinders of $(r_0,q_0)$,  with the following property:  for any combinatorial type of triangulation $\Delta$ that appears as the
Delaunay triangulation $\Delta(t_n)$ of  $(r(t_n),q(t_n))$ for a sequence $t_n\to\infty$,
there exists some $T_i$ combinatorially equivalent to $\Delta$ on the complement of the cylinders. 
\end{proposition}
\proof  For any such $\Delta$, choose $t_1\geq 0$ to be the smallest time for which 
$\Delta(t_1)$ appears in its combinatorial equivalence class.  Now let $T_1$ be the 
pullback of $\Delta(t_1)$ by the Teichmuller map $f:r(0)\to r(t_1)$.  We remark that $T_1$ is not necessarily Delaunay with respect to the flat structure given by $q(0)$.  

We now do this for each new combinatorial class that appears along $r(t)$.  There are only finitely many such $T_i$ since there 
are only finitely many combinatorial types of triangulations with a fixed number of vertices and edges.  
\endproof

\bigskip
\noindent
{\bf Step 2 (Building the fake Teichm\"{u}ller map): }
Given $r(t)$, we will build a very efficient map $\psi$ from some $r(0)$ to
$r(t)$.  We first need the following lemma about Euclidean triangles.

\begin{lemma}[Euclidean triangle lemma]
\label{lemma:triangle}
Fix a triangle $T_0$ in the Euclidean plane. Then there is a constant $b$, depending only on $T_0$,  
with the following property: for any other
Euclidean triangle $T$  whose  shortest side has  length  at least $\epsilon$, such that  
each side has length at most $R$, and which can be
inscribed in a circle of radius $R$,  there is an affine map from
$T_0$ to $T$ which has quasiconformal dilatation at most $bR/\epsilon$.
\end{lemma}

\proof
Let $p_i, i=1,2,3$ be the vertices of $T$ on the circle arranged in counterclockwise order and assume $\bar{p_1p_2}$ is the shortest side with length 
$a_1\geq \epsilon$, $\bar{p_1,p_3}$ is the longest side with length $a_3\leq 2R$.  
Let $p_0$ the center of the circle. 
Let $\theta$ the angle at $p_3$ of the $T$.  
We claim that $$\sin(\theta)=\frac{a_1}{2R}\geq \frac{\epsilon}{2R}.$$
The first case is that the segment $\overline {p_1p_3}$ separates $p_0$ from $p_2$.
Let  $\psi_1$ be the angle at $p_0$ of the isoceles triangle with vertices at $p_0,p_1,p_2$.
Let $\psi_2$ the angle at $p_0$ of the isoceles triangle with vertices $p_0,p_2,p_3$. 
Let $\psi$ the angle at $p_3$ of the isoceles triangle with vertices $p_0,p_1,p_3$.  
Since this triangle is isoceles, we have $$2\psi=(\pi-(\psi_1+\psi_2)).$$
Since the triangle with vertices at $p_0,p_2,p_3$ is isoceles, 
we have $$2(\psi+\theta)=\pi-\psi_2.$$
Subtracting we get $$\theta=\psi_1/2,$$
proving the claim. A similar analysis holds if $\overline {p_1p_3}$ does not separate $p_0$ from $p_2$. 

Similarly we have $\theta'$,  the angle of $T$ at $p_1$, is given by $$\sin(\theta')=\frac{a_2}{2R}\geq \sin(\theta),$$
where $a_2$ is length of the side  $\bar{p_2p_3}$. 
Now let $h$ be the height of the triangle $T$ with vertex $p_2$ and opposite side length $a_3$.
It divides $T$ into a pair of triangles $T_1,T_2$ with bases $x_1,x_2$ along $\overline{p_3p_1}$ and angles $\theta,\theta'$. Since $\theta'\geq\theta$ 
we have $$x_2/h\leq x_1/h=\cot (\theta)\leq 1/\sin(\theta)\leq 2R/\epsilon.$$ 
Thus if we double the triangles along the hypotenuse we find their moduli are bounded by $2R/\epsilon$ and so the affine map to a standard isoceles right triangle has dilatation bounded in terms of $R/\epsilon$. 
\endproof

\begin{proposition}
For $t$ sufficiently large, there exists a map $\psi:r(0)\to r(t)$ which is
at most an  $e^{2t}$-quasiconformal map and which is not the Teichm\"{u}ller map.
\end{proposition}

\proof
For any $t$ sufficiently large,  choose $T_i$ such that the $(r(0),q(0))$ triangulation $T_i$ 
described in Proposition \ref{proposition:finite:triang} is 
combinatorially equivalent to the Delaunay triangulation $\Delta(t)$ on 
$r(t)$, say via a homeomorphism $h:r(0)\to r(t)$.

Let $a$ be the length of the shortest vertical saddle connection of
$(r(0),q(0))$.  We now build the map $\psi: r(0)\to r(t)$.  On each vertical
cylinder, $\psi$ will be the linear map of least quasiconformal
dilatation, which is $e^{2t}$.  Notice this is the map that agrees with the
Teichmuller map $f:r(0)\to r(t)$ on the cylinder. Extend the map to the
obvious linear map on the critical graph $\Gamma(0)$.  We are left with having to
define $\psi$ on the nonempty collection of complementary minimal components of $r(0)\setminus \Gamma(0)$.  

The homeomorphism $h$ gives a bijective mapping between the set of
triangles of $T_i$ and those of $\Delta(t)$.  For each triangle $P$ of $T_i$, each of $P$ and $h(P)$ has 
a given Euclidean structure.  Define $\psi$ to be the unique affine map wihich identifies edges in 
the same combinatorial way as $h$ does.

Let $a$ be the length of the shortest edge in the critical graph $\Gamma$.
By property (4) of Proposition \ref{proposition:delaunay} applied to $\Delta(t)$, we can apply
Lemma \ref{lemma:triangle} with 
$\epsilon\geq ae^{-t}$ and $R=e^{t}/c(t)$ 
to conclude that on the union of the 
interiors of the triangles which are not in any vertical cylinder, the
pointwise quasiconformal dilatation is at most $$\frac{\displaystyle e^tb}{\displaystyle ac(t)e^{-t}}
=\frac{\displaystyle be^{2t}}{\displaystyle ac(t)}.$$
Note that since there are only finitely many $T_i$, the constant $b$ is universal.

Since $c(t)\to \infty$, this number can be taken to be smaller than
$e^{2t}$.  Note that with quasiconformal maps we only need to check
dilatation on a set of full measure, since quasiconformal dilatation is
an $L^\infty$ norm.


Thus the (global) dilatation $K(\psi)$, as a supremum of the dilatation
over all points on the surface,  equals $e^{2t}$, but $\psi$ is not the
Teichm\"{u}ller map since the dilatation is not constant. 
Namely, it is strictly smaller than $e^{2t}$ for any point in the 
minimal component. 
\endproof

\bigskip
\noindent
{\bf Step 3 (The trick): }
Since $\psi$ is not the Teichm\"{u}ller map, there is a Teichm\"{u}ller
map $\Phi:r(0)\to r(t)$ in the same homotopy class of $\psi$, with
dilatation strictly smaller than that of $\psi$, which is $e^{2t}$.  Hence
the distance in moduli space from $r(0)$ to $r(t)$ is strictly less than
$\frac{1}{2}\log e^{2t}=t$, and we are done.

\subsection{ADM rays}

Our goal in this subsection is to prove the following.

\begin{theorem}
A Teichmuller geodesic $r(t)$ determined by $(X_0,q_0)$ is ADM 
if and only if it is mixed Strebel.
\end{theorem}

\begin{proof}
Suppose $(X_0,q_0)$ is mixed Strebel. Let $C$ be a cylinder with
modulus $M$ in the homotopy class of some $\beta$.   Let 
$$b:=\inf \{\Ext_{X_0}(\alpha): \mbox{$\alpha$ is a simple closed curve}\}>0$$ 

On $r(t)$ the image of $C$ has modulus $e^{2t}M$.  By the geometric definition
of extremal length, the extremal length of $\beta$ on $r(t)$ is at most
$e^{-2t}/M$.  By Kerckhoff's distance formula (Theorem \ref{theorem:kerckhoff:formula} above), for any $\phi\in\Mod(S)$, 
$$d_{\Teich(S)}(\phi(r(0)),r(t))\geq \frac{1}{2}
\log\frac{Mb}{e^{-2t}}=t+\frac{1}{2}\log M+\frac{1}{2}\log b.$$ We have thus proved with $C=-\frac{1}{2}\log M-\frac{1}{2}\log b$ that mixed Strebel implies ADM.

Now assume that $r(t)$ is ADM.  We need to show that  $(X_0,q_0)$ is mixed Strebel.  We argue by 
contradiction: assume that $q_0$ has no vertical cylinder.

Since $r(t)$ is ADM, it cannot return to any compact set in $\M(S)$ for arbitrarily large times.  
Therefore, for sufficiently large $t$, there is a nonempty maximal collection 
$\beta_1(t),\ldots,\beta_n(t)$ of simple closed curves whose hyperbolic length is less than some 
fixed $\epsilon$, the Margulis constant.  
%

We have $$|\beta_j(t)|\geq e^{-t}
|\beta_j(t)|^{\rm vert}_0\geq ce^{-t}|\beta_j(t)|_0,$$  for some fixed $c>0$.  By Theorems 4.5 and 4.6 of \cite{Mi2}, since by assumption $(X_0,q_0)$ has 
no vertical cylinder, we have that for some fixed $\delta>0,\delta'>0$:
\begin{equation}
\label{eq:b2}
\Ext_{r(t)} (\beta_j(t))\geq \frac{\delta}{-\log |\beta_j(t)|_t}\geq\frac{\delta}{t-\log(|\beta_j(t)|_{0})}\geq \frac{\delta'}{t},
\end{equation} for $t$ sufficiently large. 

By a theorem of Maskit (see \cite{Mas}), 
the ratio of the hyperbolic length of $\beta_j(t)$ to its extremal length tends to $1$ as $t\to\infty$, so we can assume that the hyperbolic lengths of $\beta_j$ satisfy the same lower bounds.

Now fix a collection of uniformly bounded length curves $\gamma_1,\ldots, \gamma_n$ on $X_0$ combinatorially equivalent to the collection of $\beta_i(t)$, which means that the is an element $\phi(t)$ of the mapping class group taking the $\beta_i(t)$ to the $\gamma_i$. 
Since the curves in any complementary component $Y$  of the $\beta_i$ have length bounded below, we can further choose $\phi$ on $Y$  so that for  any curve of
$\phi(Y)$ the extremal lengths on $\phi(r(t))$ and  $X_0$ have bounded ratio.

By moving  a bounded Teichmuller distance we can shorten the $\gamma_i$ so that they have fixed length $\epsilon$.   We can now apply the Minsky product theorem (see \cite{Mi1})
to find constants $C_1,C_2$ such that  
that $$d_{\Teich(S)}(X_0,\phi(r(t)))\leq
\max_j\{\frac{1}{2}\log \frac{C_1}{\Ext_{r(t)}(\beta_j(t))}\}+C_2$$ which by
 (\ref{eq:b2}) is at most $$\log t-\log \delta'+\log
C_1+C_2$$
for $t$ sufficiently large.  Thus $r(t)$ is not almost length
minimizing.
\end{proof}

\subsection{Asymptote classes of EDM rays}
\label{section:asymptote}

We say that two rays $r,r'$ are {\em asymptotic} if there is a choice of basepoints $r(0), r'(0)$ so that 
$\lim_{t\to\infty}d(r(t),r'(t))\to 0$.  In this section we 
determine the asymptote classes of EDM rays.  We will then use these rays 
in Section \ref{section:dm} to compactify $\M(S)$.

\begin{definition}[Endpoint of a ray]
Let $(X,q)$ be a Strebel differential  with maximal cylinders $C_1,\ldots, C_p$, determining a ray 
$r:[0,\infty)\to\Teich(S)$. Cut each $C_i$ along a circle and glue into each side of the cut an infinite cylinder. 
The resulting surface with punctures $\hat X$ is the {\em endpoint} of $r$, denoted $r(\infty)$. It carries a quadratic differential $q(\infty)$ with double poles at the punctures, with  equal residues,  such that the vertical trajectories are closed leaves isotopic to the punctures. 

The surface $\hat{X}$ can be considered as an element of the product of Teichmuller  spaces of its connected components.  We denote this moduli space, or product of moduli spaces, which we endow with the 
$\sup$ metric, by $\Teich(\hat{X})$.
\end{definition}

We note that $\hat X$ and $q(\infty)$ do not depend on where $C_i$ is cut.   The following definition is due to Kerckhoff \cite{Ke}.

\begin{definition}[Modularly equivalent differentials]
Suppose that $(X,q),(X',q')$ are Strebel differentials with maximal cylinders  $C_1, C_2,\ldots , C_p$ and $C_1',\ldots,  C_r'$ respectively.  We say that these differentials are {\em modularly equivalent} if 
each of the following holds:
\begin{enumerate}
\item $p=r$.  
\item After reindexing, up to the action of there is an element $\Mod(S)$,  for each $i$, $\phi(C_i)$ is homotopic to $C_i'$.
\item There exists $\lambda>0$ so that $Mod(C_i)=\lambda Mod(C_i')$ for each $i$.
\end{enumerate}
\end{definition}

Suppose a pair of rays $r,r'$ are modularly equivalent. Since the moduli change by a fixed factor along rays, we can choose basepoints $r(0),r'(0)$ so that the cylinders have the same moduli at the basepoints,  and  define 
$$d(r,r')=\lim_{t\to\infty} d_{\M(S)}(r(t),r'(t))$$ if the limit exists.

\begin{theorem}
\label{thm:distance}
With the notation as above, suppose that $r$ and $r'$ are modularly equivalent.  Then 
$d(r,r')$ exists and $d(r,r')=d_{\M(\hat{X})}(r(\infty),r'(\infty))$.  
\end{theorem}

Assuming Theorem \ref{thm:distance} for the moment, we have the following.

\begin{corollary}
\label{cor:asym}
Two rays $r,r'$ are asymptotic if and only if they are modularly equivalent and they have the same endpoints $r(\infty)=r'(\infty)$.
\end{corollary}

This corollary was proven by Kerckhoff \cite{Ke} in the case of a maximal collection of cylinders.

\begin{proof}[of Corollary \ref{cor:asym}]
The ``if'' direction follows immediately from Theorem \ref{thm:distance}.  For the ``only if'' direction, 
we first note that the hypothesis implies that for each $n$ sufficiently large, there is a sequence 
of  $(1+o(1))$-quasiconformal maps $f_n:r(n)\to r'(n)$.  Since uniformly quasiconformal maps form a normal family (see, e.g., \cite{Hu}, Theorem 4.4.1) and $r(n),r'(n)$ converge to $r(\infty),r'(\infty)$, 
there is a subsequence of $\{f_n\}$ which converges to 
a conformal map $f_\infty:r(\infty)\
\to r'(\infty)$, so that $r(\infty)=r'(\infty)$.  Modular equivalence of $r$ and $r'$ follows immediately from (1) of Lemma \ref{lem:short1} and Kerckhoff's distance formula (Theorem \ref{theorem:kerckhoff:formula}).  
\end{proof}

We now begin the proof of Theorem \ref{thm:distance}.  

\begin{proof}
We first note that, exactly as in the proof of Corollary \ref{cor:asym}, we have $$d_{\M(\hat{X})}(r(\infty),r'(\infty))\leq\liminf_{t\to\infty} d_{\M(S)}(r(t),r'(t)).$$ 

To prove the opposite inequality we first need the following lemma.

\begin{lemma}
\label{lem:close}
Suppose $\epsilon>0$ is given.   Let $C_1,C_2$ be  Euclidean cylinders with heights $R_1,R_2$ and circumference $1$. Now in coordinates $(x,y)$ in the upper half-space model $\hyp^2$ of the hyperbolic plane, given any $n\in\Z$ we let $z_1=(0,R_1)$ and 
$z_2=(n,R_2)$ be points in $\hyp^2$.  Let 
$$d_0:=d_{\hyp^2}(z_1,z_2).$$  Let $p_1,q_1$ marked points on the boundary of $C_1$ assumed to be at 
$(0,0)$ and $(0,R_1)$.  Let $p_2,q_2$ marked points on the boundary of $C_2$ at $(0,0)$ and $(\alpha,R_2)$ in polar coordinates $(\theta,h)$ on $C_2$.
 Let $f(\theta)$ be a real analytic function defined from the base $h=0$ of $C_1$ to the base of $C_2$ such that $f(0)=0$ and $$\sup_\theta|f'(\theta)-1|\leq \epsilon.$$   Let  $\gamma_1$ be the vertical line in $C_1$ joining $p_1$ to $q_1$.  Let $\beta$ be the Euclidean geodesic 
 joining $(0,0)$ to $(\alpha,R_2)$ in $C_2$. Let $\gamma_2$ be the local geodesic in the relative homotopy class of $\beta$ twisted $n$ times about the core curve of $C_2$.  Then for $R_1,R_2$ large enough,  there   is a 
 $(1+O(\epsilon))e^{2d_0}$-quasiconformal map $F:C_1\to C_2$ 
such that 
\begin{itemize}
\item $F(\theta,0)=(f(\theta),0)$.
\item  $F(q_1)=q_2$.
\item  $F(\gamma_1)$ is homotopic to $\gamma_2$ relative to the boundary of $C_1$. 
\end{itemize}
\end{lemma}
\begin{proof}
Define $F=(F_1,F_2)$ by $$F(\theta,h)=((1-\frac{h}{R_1})f(\theta)+\frac{h(\theta+\alpha+n)}{R_1},\frac{hR_2}{R_1});$$
the first coordinate taken modulo $1$.  We have $F(\theta,0)=(f(\theta),0)$ and 
$F(0,R_1)=(\alpha,R_2)$ and 
$F(\gamma_1)=\gamma_2$. 
We compute 
$$\partial F_1/\partial \theta=\frac{h}{R_1}(1-f'(\theta))+f'(\theta)$$
$$\partial F_2/\partial h= \frac{R_2}{R_1}$$ $$\partial F_1/\partial h=\frac{1}{R_1}(\theta+\alpha+n-f(\theta))$$ and $$\partial F_2/\partial \theta=0.$$
So we have 
$$|\partial F_1/\partial \theta-1|<2\epsilon$$ and 
for $R_1$ sufficiently large we have $$|\partial F_1/\partial h-n/R_1|\leq \epsilon.$$ 
Thus 
$$
|\text{Jac}(F)-\left(
\begin{array}{ll}
1&n/R_1\\
0&R_2/R_1
\end{array}\right)|=O(\epsilon),
$$
where Jac stands for Jacobian.  Note that the above linear map is the Teichmuller map taking the marked torus spanned by $\{(1,0),(0,R_1)\}$ to the marked torus spanned by 
$\{(1,0), (n,R_2)\}$. These tori correspond to the given points in $\hyp^2$ and therefore the dilatation of the linear map is precisely $e^{2d_0}$, as claimed.  \end{proof}

\medskip
\noindent

We also need the following lemma.  

\begin{lemma}
\label{lem:punctures}
Let $g:\hat X\to\hat X'$ a Teichmuller map with dilatation $K_0$. Given $\epsilon>0$, there is a 
$(K_0+\epsilon)$-quasiconformal map $f:\hat X\to\hat X'$ which is conformal in a neighborhood of 
of the punctures.
\end{lemma}
\begin{proof}
Let $\mu$ be the dilatation of  $g$. 
For any small neighborhood $0<|z|<|t|$ of the punctures, let $\mu_t$ be the Beltrami differential which is $0$ in $0<|z|<|t|$ and $\mu$ in the complement.
For some surface $\hat{X}_t$, there is a $K_0$-quasiconformal map $f_t:\hat X\to \hat{X}_t$ with dilatation $\mu_t$; in particular $f_t$ is conformal in $0<|z|<|t|$.   As $t\to 0$, $\mu_t\to \mu$ and therefore  
$\lim_{t\to 0}f_t=g$ and so $\lim_{t\to 0}\hat{X}_t= \hat X'$.   Choose a nonempty open set $\V$ on $\hat{X}$.  We can find a collection of  Beltrami differentials supported in $\V$ that form a basis for the tangent space to $\Teich$ at $\hat{X'}$.  This implies that  for $t$ small enough we can find a $(1+\epsilon)$-quasiconformal map $h_t:\hat{X}_t\to \hat X'$
which is conformal in a neighborhood of the punctures.
 Our desired map is $f=h_t\circ f_t$.
\end{proof}

Now  we begin the proof of the bound $$\limsup_{t\to\infty}d_{\M(S)}(r(t),r'(t))\leq d_{\M(\hat{X})}(r(\infty),r'(\infty)).$$
Let $p_i,q_i, i=1,\ldots, p$   be the paired punctures on $r(\infty)$, and let 
 $z_i$ be the coordinate at $p_i$  so that for some $a_i>0$, $$q(\infty)=\frac{a_i^2}{z_i^2}dz_i^2,$$ we have a similar coordinate in a neighborhood of $q_i$. 
Let $\zeta_i$ the corresponding coordinate for $q'(\infty)$ on $r'(\infty)$
in a neighborhood of $p_i'$ so that $$q'(\infty)=\frac{b_i^2}{\zeta^2_i}d\zeta_i^2.$$  
 Circles in these coordinates are  vertical leaves for $q(\infty)$ and $q'(\infty)$ 
and have lengths $2\pi a_i$ and $2\pi b_i$ respectively. 
For some $\delta_j(t)$ we recover the surfaces along $r(t)$ by removing  punctured discs of radius 
$\delta_i^{1/2}(t)$ around  $p_i$ and $q_i$ 
and glueing the resulting surfaces along their boundary. 
We have $$\lim_{t\to 0}\delta_j(t)=0.$$
We have a similar picture for $r'$ with corresponding $\delta_i'^{1/2}(t)$.  The assumption  that $r,r'$ are modularly equivalent means that for each $\delta_i$ there is $\delta_i'$, such that  the resulting cylinders $A_i,A_i'$ on $r(t),r'(t)$ have the same modulus. 
For convenience we drop the subscript $i$.

Let $K=e^{d_{\M(\hat{X})}(r(\infty),r'(\infty))}$. 
Given $\epsilon$,  let  $F_2:r'(\infty)\to r(\infty)$ be the $(K+\epsilon)$-quasiconformal map given by Lemma~\ref{lem:punctures} that is  conformal in a neighorhood of all of the  punctures. 
We may take a fixed $\kappa'$ so that $F_2$ is conformal inside the circle of radius $\kappa'$ inside each punctured disc.   This means that we can take $\zeta$ as a conformal coordinate in a neighborhood of the puncture on 
$r(\infty)$ and so the map $F_2$ is the identity on the circle
$|\zeta|=\kappa'$ in these coordinates.

Consider the annulus  $B'\subset A'$ defined by 
$$B'=\{\zeta:|\delta'^{1/2}|<|\zeta|<\kappa'\}.$$ 
Consider also the annulus $B\subset r(\infty)$ which in the $z$ plane is bounded by the circle of radius $|\delta^{1/2}|$
and the curve  $\omega$ which is the image under $F_2$ of the circle of radius $\kappa'$. 
In the $\zeta$ coordinates on $r(\infty)$,  $B$ is bounded by the circle $|\zeta|=\kappa'$ and an analytic curve  $\gamma$ 
which is the image under the holomorphic change of coordinate map  $\zeta=\zeta(z)$  of the circle  of radius $|\delta^{1/2}|$.   

Since $\kappa'$ is fixed, we have $$\lim_{\delta'\to 0}\frac{\text{Mod}(B')}{\text{Mod}(A')}=1$$ and since $\omega$ is fixed, $$\lim_{\delta\to 0}\frac{\text{Mod}(B)}{\text{Mod}(A)}=1.$$ Since $\text{Mod}(A)=\text{Mod}(A')$ we therefore have 
 \begin{equation}
\label{eqn:mod}
\lim_{\delta\to 0}\frac {\text{Mod}(B')}{\text{Mod}(B)}\to 1.
\end{equation}

For small enough $\delta$ we
 wish to find a $(1+O(\epsilon))$ quasiconformal map $F_1$ from $B'$ to $B$ such that 
\begin{itemize}
\item for $\zeta=\delta'^{1/2} e^{i\theta}, z=F_1(\zeta)=\delta^{1/2} e^{i\theta}$
\item for $|\zeta|=\kappa'$,  $F_1(\zeta)=\zeta$. 
\end{itemize}

In other words, the desired $F_1$ is the identity on the circle of radius $\kappa'$ and takes the circle of radius $\delta'^{1/2}$ in the $\zeta$ coordinates to the circle of radius $\delta^{1/2}$ in the $z$ coordinates. 
We also find a corresponding map $F_1$ for  neighborhoods of the punctures $q_i,q_i'$.  We then will glue these maps $F_1$ along the circle of radius $\delta'^{1/2}$  together to give a $1+O(\epsilon)$ quasiconformal map, again denoted $F_1$,  on the glued annulus to the annulus found by gluing along the circle of radius $\delta^{1/2}$ in the $z$ coordinates.  We then glue $F_1$ to $F_2$ along the circles of radius $\kappa'$ to give a 
$(K+O(\epsilon))$-quasiconformal map from $r'(t)$ to $r(t)$.

We now find the map $F_1$.
By (\ref{eqn:mod}) for all sufficiently small $\delta$, $$|\frac{\text{Mod}(B')}{\text{Mod}(B)}-1|\leq\epsilon/2$$
Find a conformal map $h_\delta(z)$ from $B$ to a round annulus 
$$B_1=\{w:\delta''^{1/2}<|w|<\kappa'\}$$ with the normalization that  
 $h_\delta(\kappa')=\kappa'$.     
The composition $$\zeta=\delta'^{1/2}e^{i\theta}\to z=\delta^{1/2}e^{i\theta}\to h_\delta(z)$$ is a map 
$w=f_\delta(\zeta)$ from  the circle of radius $\delta'^{1/2}$ in the $\zeta$ plane to the circle of radius $\delta''^{1/2}$ in the $w$ plane.  Similarly  we have a  map $w=g_\delta(\zeta)$ from  the circle of radius $\kappa'$ in the $\zeta$-plane to the circle of radius $\kappa'$ in the $w$-plane.  These two maps can be thought of as boundary maps of $B'$ to $B_1$.

We wish to show that, as $\delta\to 0$,  we have $|f_\delta'(\zeta)-(\delta^{''}/\delta')^{1/2}|\to 0$ and $|g_\delta'(\zeta)-1|\to 0$.  In that case after mapping the annuli $B_1,B'$ to flat cylinders with base $0$, circumference $1$  and heights $R_1,R_2$ respectively, by a logarithm map,  the induced maps on the top and bottom of the cylinders have derivatives
 almost constantly $1$. Since the ratio of moduli has limit $1$,  we then can apply Lemma~\ref{lem:close} with $R_1/R_2\to 1$ and $n=0$.

We now show the desired above limits hold. 
Considering  $B$ as an annulus in the $\zeta$ coordinates,  with outer boundary the  fixed circle $|\zeta|=\kappa'$,   
as $\delta\to 0$, the conformal maps  $h_\delta$ converge to a conformal self map of the punctured disc $0<|\zeta|<\kappa'$. 
It extends to a conformal map   taking $0$ to $0$.  The only such conformal maps are rotations.  But by our normalization of the $h_\delta$'s to fix a point, that map must be the identity.  Thus as $\delta\to 0$, the  maps $h_\delta$  converge uniformly to the identity, and therefore $g'_\delta$ converges uniformly to $1$ on the circle of radius $\kappa'$.  

By replacing $z$ with $z/\delta^{1/2}$, and $w$ with $w/\delta''^{1/2}$ we also can consider $h_\delta$  as a map from the annulus $B$ in the $z$ plane with inner boundary the unit circle, to $B_1$, another  annulus with inner boundary the unit circle. 
As $\delta\to 0$, $h_\delta$  converges to a conformal map of the exterior of the unit disc to the exterior of the unit disc, taking $\infty$ to $\infty$.   The limiting conformal map is  therefore again the identity.  Thus the map $h_\delta$ from the circle of radius $\delta^{1/2}$ to the circle of radius $\delta''^{1/2}$ in the $w$ plane has derivative approaching $(\delta''/\delta)^{1/2}$ as $\delta\to 0$. Since the map from the circle of radius $\delta'^{1/2}$
in the $\zeta$ plane to the circle of radius $\delta^{1/2}$ in the $z$ plane has derivative $(\delta/\delta')^{1/2}$, applying the chain rule  
 the composition $f_\delta$ has derivative converging to  $(\delta''/\delta')^{1/2}$ as $\delta\to 0$.  We are now in a position to apply Lemma~\ref{lem:close}.
This completes the proof.  
\end{proof}

\section{The iterated EDM ray space and the Deligne-Mumford compactification}
\label{section:irdm}

In this section we introduce a functor $X\mapsto \bar{X}^{ir}$
defined on a certain collection of metric spaces $X$.
The space $\bar{X}^{ir}$ will be constructed via certain
equivalence classes of EDM rays, and will have the structure of a
metric stratified space (see below).  We will then prove that
this functor applied to $\M(S)$ produces the Deligne-Mumford
compactification $\DM$; that is, we will find a 
stratification-preserving homeomorphism from $\bar{\M(S)}^{ir}$
to the Delgine-Mumford compactification $\DM$ which is an
isometry on each stratum.

\subsection{The iterated EDM ray space}
\label{section:irc}

Before defining $\bar{X}^{ir}$, we will have to deal with a technical issue.  The boundary pieces of $\DM$ are naturally products of smaller moduli spaces.  We will need to canonically pick out the factors in such products by studying uniqueness of product decompositions.  Unfortunately, the 
fact that $\M(S)$ has orbifold points slightly complicates matters, as we will now see.

A metric space $Y$ is said to have the {\em unique local geodesic property} if for every $y\in Y$ there is a neighborhood $U$ of $y$ with the property that
 any two  points in $U$ can be connected by a 
unique geodesic in $U$.  It is well-known that $\Teich(S)$ has the unique local geodesic property.  
It follows easily from the proper discontinuity of the action of $\Mod(S)$ on $\Teich(S)$ that $\M(S)$ has this property in the complement of its orbifold locus.  However, for points $s\in\M(S)$ in the orbifold locus, this is not true: every neighborhood of $s$ in $\M(S)$ has some pair of points $x,y$ so that the number $n(x,y)$ of (globally length minimizing) geodesics from $x$ to $y$ is greater than $1$.  Since there is a uniform bound (of $84(g-1)$) of the order of any group stabilizing any point of $\Teich(S)$, it follows that there is a uniform upper bound for $n(x,y)$ for any $x,y\in\M(S)$.

\begin{theorem}[Uniqueness of product decomposition]
\label{theorem:uniqueness3}
For each $1\leq i\leq m$, let $Y_i$ be a connected metric space, not equal to a point, with the following property: 
\begin{enumerate}
\item The complement of the 
set of points $S_i\subset Y_i$ without the unique local geodesic property is open and dense in $Y_i$, and 
\item there exists $N_i\geq 1$ so that for all $x,y\in Y_i$, 
the number $n_i(x,y)$ of (globally length-minimizing) geodesics in $Y_i$ from $x$ to $y$ is 
at most $N_i$.
\end{enumerate}

Let $Z=Y_1\times Y_2\ldots\times Y_n$, endowed with the $\sup$ metric.  Then given any other way of writing $Z=X_1\times\cdots X_m$ with the $\sup $ metric, it must be that $m=n$ and, after perhaps permuting factors, 
$X_i=Y_i$ for all $i$. 
 \end{theorem}

As the proof of Theorem \ref{theorem:uniqueness3} is independent of the rest of this paper, we leave it for the Appendix (Section \ref{section:appendix}) below.  One key ingredient is a recent theorem of Malone \cite{Mal}.

As discussed above, $Y_i=\M(S)$ satisfies the hypotheses of Theorem \ref{theorem:uniqueness3}.   In this case the set $S_i$ is precisely the orbifold locus of $\M(S)$.

\bigskip
Now consider a metric space $(X,d)$ with $X=X_1\times\ldots X_m$ ( possibly with $m=1$).  Assume that $(X,d)$ satisfies the hypotheses of Theorem \ref{theorem:uniqueness3}.
We will consider rays in each factor. 

\begin{definition}[Isolated rays]
We say that a ray $r$  is {\em isolated} if the following two properties hold
\begin{enumerate}
\item There is a factor $X_j$ such that $r\subset X_j$ and $r$ is an EDM ray in $X_j$.    
\item For every $p\in X_j$, the set of asymptote classes of EDM rays $[r']\subset X_j$ which are a bounded distance from $r$, and which have some representative passing through $p$, is countable.
\end{enumerate}
\end{definition}

We will now define a space $\bar{X}^{ir}$ inductively, 
building it inductively,  stratum by stratum.  The level $k$ stratum will be 
denoted $D_k(X)$. 

\medskip
Henceforth every metric space $(Y,d)$ that appears as a factor in a product  will be assumed to have the following 
three properties:

\bigskip
\noindent
{\bf Standing Assumption I (Limits exist): } For any two isolated EDM rays $r_1,r_2$ in $Y$ that are a bounded distance apart, 
there are initial points $r_1(0),r_2(0)$ such that 
$\lim_{t\to\infty} d(r_1(t),r_2(t))$ exists and is a minimum among all choices of basepoints.  

\bigskip
\noindent
{\bf Standing Assumption II (Asymptotes are uniformly asymptotic):} For any $\epsilon>0$,   any asymptote class of isolated EDM rays $[r]$,  any representative $r$ of $[r]$, and any choice of  basepoint $r(0)$,
 there is a $T=T(r,r(0), \epsilon)$  such that for any such asymptotic 
 pairs $r,r'$ the rays $r([T,\infty))$ and 
$r'([T',\infty))$ are within Hausdorff  distance $\epsilon$ of each other. 

\bigskip
\noindent
{\bf Standing Assumption III (Almost locally unique geodesics): }$Y$ satisfies the hypotheses (and hence the conclusions) of Theorem \ref{theorem:uniqueness3}

\bigskip
If a metric space  $X$ contains isolated rays, we consider the set $\text{Asy}(X)$ of all asymptote 
classes of isolated EDM rays $[r]$ in $X$.  With Standing Assumption I in hand, 
we can endow $\text{Asy}(X)$ with a distance function via $d_{\text{asy}}([r_1],[r_2])=\lim_{t\to\infty} d(r_1(t),r_2(t))$ for choice of basepoints that minimizes this limit.  It is easy to check that this defines a metric.
 
 
\medskip

\medskip
Let $(D_0(X), d_0):=(X,d)$.  

\medskip
\noindent
{\bf Step 1 (Inductive step): } Suppose we are given the metric space $D_k(X)$, written as a product of factors $X_1\times \ldots \times X_m$ with the metric  $d_k(\cdot,\cdot)$, where $d_k$ is the $\sup$ of the metrics $d^j$ of the factors.  
Remove each factor that is a point. If none of the factors  $X_j$  contains isolated EDM rays, define $D_{m}(X)=\emptyset$ for all $m>k$ and stop the inductive process.   If some factor $X_j$ contains isolated rays then we 
 set $$D_{k+1}^j(X)=X_1\times \ldots \times X_{j-1}\times \text{Asy}(X_j)\times X_{j+1}\times \ldots \times X_m.$$
We can endow $D_{k+1}^j(X)$ with a distance function $d_{k+1}^j$ as the sup metric on the factors.    
 From Standing Assumption III, we have that if $\text{Asy}(X_j)$ is a product, then it can be written uniquely as a product.  Thus, given the product representation of $D_k(X)$, we have a unique product representation of $D_{k+1}^j(X)$.  

Note also that if two points in $D_{k+1}^j(X)$ have an infinite distance from each other, then they are in different components of 
$D_{k+1}^j(X)$.  
We then set $$D_{k+1}(X)=\sqcup_{j=1}^m D_{k+1}^j(X)$$ with metric $d_{k+1}$ which is the corresponding metric $d_{k+1}^j$ on each term in the disjoint union. 
\medskip

\noindent
{\bf Step 2 (Topology): }
We will inductively define a topology on the disjoint union $Y:=\cup_{j=0}^\infty D_j(X)$, as follows.

Using Standing Assumption II,  for every $[r_0]\in \text{Asy}(X_j)$ and every $\epsilon>0$ we can define 
an $\epsilon$-neighborhood  $V_\epsilon([r_0])$ of $[r_0]$ in $\text{
Asy}(X_j)\cup X_j$.   Consider the set of equivalence classes of isolated rays $[r]\in \text{Asy}(X_j)$  such that $d^j([r],[r_0])<\epsilon$  
and set  $V_\epsilon^j([r_0])$ to be the union of the set of such rays  and the following set. 
For each such ray $[r]$ and each $r\in [r]$ include in $V_\epsilon^j([r_0])$  the set 
$\{r(t): t\geq T(r,r(0),\epsilon)\}$. 

We are now ready to define the topology. 
\begin{definition}
\label{df:top1}
Let $j\geq 0$.  Suppose $\vec{x}(n)$ is a sequence in $D_k(X)$ and $(\vec{x},[r])\in D_{k+1}^j(X)$.
   We say $\vec{x}(n)\to (\vec{x},[r])$ if there exists  $t_n\to\infty$ 
such that 
\begin{enumerate}
\item  for $i\neq j$, $\lim_{n\to\infty} d^i(x_i(n),x_i)= 0$ 
\item  $\lim_{n\to\infty} d^j(x_j(n),r(t_n))=0$ for some representative $r$ of $[r]$.
\end{enumerate}
\end{definition}

Now suppose inductively for each $k,m$, and for each sequence $\vec{x}(n)\in D_k(X)$, and $y\in D_{k+m}(X)$ we have defined what it means for $\vec{x}(n)$ to converge to $y$.
\begin{definition}
\label{df:top2}
Suppose $\vec{x}(n)\in D_k(X)$ and $z\in D_{k+m+1}(X)$. We say $\vec{x}(n)\to z$ if there exists $j$, points  $(\vec{x'}(n),[r_n])\in D_{k+1}^j(X)$, a sequence $\epsilon_n\to 0$, representatives $r_n$   and times $t_n$ such that  
\begin{enumerate}
\item $\lim_{n\to\infty} d^i(x_i(n),x'_i(n))=0$ for $i\neq j$.
\item $\lim_{n\to\infty}d^j(x_j(n),r_n(t_n))=0$.
\item $r_n(t_n)\in V_{\epsilon_n}([r_n])$.
\item $\lim_{n\to\infty} (\vec{x'}(n),[r_n])=z$.
\end{enumerate}
\end{definition}
The first condition just says that one has convergence in the factors where one is not considering isolated rays. Notice the last condition inductively makes sense since $(\vec{x'}(n),[r_n])\in D_{k+1}(X)$ and $z\in D_{k+m+1}(X)$ and $k+m+1-(k+1)=m$.

\smallskip We thus obtain a topological space which is stratified
by $\{D_k(X)\}$, and in fact each stratum is a metric space (by Standing Assumption I).  Note that 
$X$ is open and dense in $Y$.  We are actually interested in a somewhat simpler space, obtained as a certain quotient of $Y$, as
follows.

\medskip
\noindent
{\bf Step 3 (Identifications): }  The space $Y$ provides a natural ``boundary'' for $X$, although the 
construction may give multiple copies of the same boundary component.  To remedy this, we will 
identify points that  ``should'' be distance zero from each other.   In some sense this is like 
Cauchy's scheme for completing metric spaces.  

We make no identifications of points in $D_0(X)$. Now suppose inductively we have made identifications of points in $D_j(X)$ for all $j\leq k$ and $P,Q\in D_{k+1}(X)$. 
\begin{definition}
We say $P\sim Q$ if there exist sequences $x_n,y_n$ in the same component of $D_{k-1}(X)$ such that 
\begin{enumerate}
\item $\lim_{n\to\infty} x_n=P$ and $\lim_{n\to\infty} y_n=Q$.
\item $\lim_{n\to\infty}d_{k-1}(x_n,y_n)=0$.
\end{enumerate}
\end{definition}
This is clearly an equivalence relation.  We denote  
 the  quotient space of $Y$ by this equivalence relation  by $\bar{X}^{ir}$, and call it the  {\em iterated EDM ray
space} associated to $X$.  This is evidently a functor from
metric spaces (whose $D_j$'s satisfy the standing assumptions) and isometries
to metric spaces and isometries.  If $Y$ turns out to be 
a compactification of $X$, then 
since we only identified certain points in $Y\setminus X$, it
follows that $\bar{X}^{ir}$ is also a compactification of $X$.

\begin{xample}
For $X$ the upper quadrant in $\R^2=\R^+\times \R^+$ with the $\sup$ metric, $D_1$ has two components, each of which is an infinite ray. A point in one component corresponds to a vertical ray, with the distance 
function equal to the distance function between vertical rays,  i.e. the difference of their $x$ coordinates. 
The points in the other component correspond to horizontal rays, with the distance being the difference of their $y$ coordinates.   Since $D_1$ is a disjoint union of two rays,  $D_2$ consists of two points.  
The sequence $(n,n)$ converges to each of the two points in $D_2$, and so 
these points are identified.  Thus in this case $\bar{X}^{ir}$ is a closed square.
\end{xample}

\subsection{Metric stratified spaces}

We would like to keep track of structures finer than topological
type.  To do so we will need the following standard concept.

\begin{definition} A {\em stratification} of a second countable,
locally compact Hausdorff space $X$ is a locally finite partition
${\mathcal S}_X$ into open sets $S$ satisfying: \begin{enumerate}
\item Each element $S\in {\mathcal S}_X$, called a {\em stratum}, is
a connected topological space in the induced topology.  \item For
any two strata $S_1,S_2\in{\mathcal S}_X$, if $\bar{S_1}\cap
S_2\neq\emptyset$ then $\bar{S_1}\supset S_2$.  \end{enumerate} A
space $X$ with a stratification, with each stratum endowed with
the structure of a metric space, is called a {\em metric
stratified space}.  \end{definition}

Inclusion $\bar{S_1}\supset S_2$ defines a partial ordering
$S_1>S_2$ on the elements of ${\mathcal S}_X$.  The {\em depth}, or
{\em level} of a stratum $T$ is the maximal $n$ so that there is
a chain $$S_0>\cdots >S_n=T$$ with $S_i\in{\mathcal S}_X$.  Note that
since ${\mathcal S}_X$ is locally finite, any such chain is finite,
although {\it a priori} one might have strata of infinite depth.

\begin{xample}
The iterated EDM ray space $\bar{X}^{ir}$ of \S\ref{section:irc} 
has a natural stratification, where the level $k$ strata are the components of $D_k(X)$.
\end{xample}

\subsection{The Deligne-Mumford compactification}
\label{section:dm}

Deligne-Mumford \cite{DM} constructed a compactification $\DM$ of $\M(S)$, called the {\em Deligne-Mumford compactification}, which they proved is a projective variety.  As such, $\DM$ is endowed with 
the structure of a stratified space.  Bers \cite{Be} also gave a construction of $\DM$ as a stratified 
space.  Points of the level $k$ strata of $\DM$ are given by conformal structures on $k$-noded Riemann surfaces; the set of strata are parametrized by the set of combinatorial types of collections 
of nodes (see \cite{Be,DM}).  

The topology on $\DM$ is as follows.  On each stratum the topology is just that of the corresponding moduli space.  Points $X_n$ converge to some $Y$ in a lower level stratum if for every neighborhood $N$ of the union of nodes in $Y$, there is a conformal map $(Y\setminus N)\to X_n$ for $n$ sufficiently large.  We endow each stratum of $\DM$ with the corresponding Teichm\"{u}ller metric, thus giving $\DM$ the structure of a metric stratified space.

Our goal in this section is to reconstruct $\DM$ as a metric
stratified space (but not as a projective variety) as the
iterated EDM ray space $\bar{\M(S)}^{ir}$ associated to $\M(S)$.  We
therefore begin by applying the construction from the previous
subsection to $\M(S)$, endowed with the Teichm\"{u}ller metric.
.  

We characterize the isolated rays in $\M(S)$, 
and identify the metric they give on the stratum $D_1(\M(S))$.

\begin{proposition} 
\label{proposition:dm1} 
Let $S$ be a surface
of finite type.  Then a ray in $\M(S)$ is an isolated EDM ray if
and only if it is a one-cylinder Strebel ray.  Let $r$ and $r'$
be one-cylinder Strebel rays.  Suppose the cylinders of $r$ and $r'$ both have 
core curves of the same topological type as a fixed simple closed curve $\gamma$.  
Then $d_1(r,r')$ in
$D_1(\M(S))$ exists, and coincides with the Teichm\"{u}ller
distance between $r(\infty)$ and $r'(\infty)$ in the boundary
moduli space $\M({S\setminus \gamma})$.  
\end{proposition}

We remark that if the cylinder defining the Strebel ray is given
by a separating curve, then $S'$ is disconnected, and so
$\M(S\setminus \gamma)$ is itself a product of smaller moduli
spaces.

\begin{proof} 
By Theorem \ref{theorem:rays}, a ray in $\M(S)$ is
EDM if and only if it is Strebel.  By Theorem 21.7 of \cite{St},
on each Riemann surface there is a unique one-cylinder Strebel
differential in each homotopy class of simple closed curve.  There are only countably
many such homotopy classes. Moreover, given a collection of more than
one distinct homotopy class of disjoint curves, the set of
Strebel differentials with cylinders in those homotopy classes is
uncountable (again, by Theorem 21.7 of \cite{St}). Moreover by
Theorem 2 of \cite {Ma1}, any two Strebel differentials with
homotopic cylinders are a bounded distance apart.  However (again by  Theorem 21.7 of \cite{St}) they are not modularly equivalent and so these classes are not isolated. It follows
easily from Lemma~\ref{lem:short1} that each of these is an
unbounded distance from a ray defined by a one-cylinder Strebel
differential. These facts together imply that the isolated rays
coincide with the one-cylinder Strebel rays.

The fact that the set of asymptote classes of one-cylinder Strebel rays on any moduli space is homeomorphic to the moduli spaces of one smaller complexity, and 
 that the distance between one cylinder Strebel rays of the same type exists and is equal to the Teichmuller distance on the corresponding one complexity smaller moduli space, is the content of 
 Theorem~\ref{thm:distance}.  The fact that isolated EDM rays determined
 by combinatorially inequivalent curves are not bounded distance apart follows from Lemma~\ref{lem:short1}. \end{proof}

With the setup above, we can now prove the main result of this
section: that $\bar{\M(S)}^{ir}$ and $\DM$ are isomorphic as
metric stratified spaces.

\begin{theorem} 
The iterated EDM ray space 
$\bar{\M(S)}^{ir}$ associated to $\M(S)$ is homeomorphic to the
Deligne-Mumford compactification $\DM$ via a
stratification-preserving homeomorphism which is an isometry on
each stratum.
 \end{theorem}
 
\begin{proof} First recall that the set of level $k$ strata of
$\DM$ is parametrized by the set of combinatorial types of
$k$-tuples of simple closed curves on $S$, representing the
curves that are pinched to nodes.  Each level $k$ stratum
corresponding to a $k$-tuple $\{\alpha_1,\ldots ,\alpha_k\}$ is a
product of the moduli spaces of the punctured surfaces consisting
of the components of $S\setminus\{\alpha_1,\ldots ,\alpha_k\}$.
Further, we have endowed each stratum with the Teichm\"{u}ller
metric of the corresponding moduli space or, in the case of
disconnected surfaces, with the $\sup$ metric on the product of
moduli spaces.

\bigskip
\noindent
{\bf Step 1 (Defining a surjective map): }  We first define a map  $$\psi: \cup_{k=0}^\infty D_k(\M(S))\to 
\DM$$ inductively, as follows.  On $D_0(\M(S))$ we simply let $\psi$ be the identity map.  Each factor that was a point that was removed is sent to the moduli space of a three times punctured sphere which is itself a point.   By Proposition \ref{proposition:dm1}, the isolated EDM rays in $D_0(\M(S))$ are precisely the one-cylinder Strebel rays.  The equivalence classes of one-cylinder Strebel differentials correspond precisely to the topological types of simple closed curves on $S$.  By Corollary \ref{cor:asym}, the asymptote classes of one-cylinder Strebel rays $r$ correspond to the possible endpoints $r(\infty)$.  By Strebel's existence theorem (Theorem 23.5 of \cite{St})), every possible endpoint can occur, so 
that $D_1(\M(S))$ consists of all possible surfaces obtainable by pinching a single simple closed curve on $S$.  Thus $D_1(\M(S))$ is the disjoint union of moduli spaces, one for each topological type of simple closed curve.  By Theorem \ref{thm:distance}, 
the  metric $d_1$ on $D_1$ coincides with the corresponding Teichm\"{u}ller metric on each component 
of $D_1(\M(S))$.  We define 
$\psi$ on each component of $D_1(\M(S))$.  If the component is not a product we  map an asymptote class $[r]$ of rays to the corresponding endpoint $r(\infty)$.  If the component is a product, then for each factor we define $\psi$ by fixing the coordinates of the other factors  and map an asymptote class of rays in the factor to its endpoint. By the above, on each component, this map is an isometry onto 
the component of $\DM$ corresponding to the appropriate combinatorial type of simple closed curve.

Suppose now inductively that we have proven that each component of $D_k(\M(S))$ is isometric via a map $\psi$ to  a  (products of) moduli spaces, and the map is onto the collection of moduli spaces, one  for each combinatorial type of $k$-tuple of simple closed curves.  Fix any component of $D_k(\M(S))$, corresponding to a $k$-tuple $\{\alpha_1,\ldots ,\alpha_k\}$, and let $\M(S')$ be the corresponding (products of) moduli spaces $\M(S_1)\times \ldots \times \M(S_p)$, where $S'=S\setminus \{\alpha_1,\ldots ,\alpha_k\}$.  For each factor in this product we find the asymptote classes of isolated EDM rays, again given by the one cylinder Strebel differentials. We thus obtain  
components of $D_{k+1}(\M(S))$, and these 
components correspond to the possible combinatorial types of $(k+1)$-tuples obtainable from $\alpha_1,\ldots ,\alpha_k$ by adding a single simple closed curve.  We again define $\psi$ on each component by 
sending each asymptote class $[r]$ to $r(\infty)$, and if the component is a product, defining it to be the identity on the other coordinates.   As above, we see that $\psi$ is an isometry when restricted to any of the fixed components just obtained.  By Strebel's theorem again, the map is onto  all $(k+1)^{\rm st}$ strata in $\DM$.  

We have therefore inductively defined a map 
$$\psi:\bigcup_kD_k(\M(S))\to \DM$$
which we have shown to be onto (by Strebel's existence theorem), and which is an isometry when restricted to any fixed component of any fixed $D_k(\M(S))$.  

\medskip
\noindent
{\bf Step 2 (The standing assumptions hold): }  Standing Assumption I holds by the fact discussed above, that if two EDM rays are defined  
by pinching the same combinatorial type of curve then the rays have an asymptotic distance  apart, and by  the fact that  if the topological types are different then the rays are not bounded distance apart.
The latter follows from Lemma~\ref{lem:short1}

Now we show Standing Assumption II holds.
Let $[r]$ be an asymptotic class of isolated EDM ray on any moduli space with $r$ any representative. As we have seen, on the surface $r(\infty)$ there is a quadratic differential $q(\infty)$ with double poles at the paired punctures, such that the vertical trajectories are all closed curves of equal length isotopic to the punctures. Since $q(\infty)$ is the unique (up to scalar multiple) quadratic differential with this property,   any two representatives determine the same $q(\infty)$.
Since the Strebel differentials along $r$ can be reconstructed by cutting out punctured discs on $r(\infty)$   and gluing along the boundary circles of $q(\infty)$, the ray $r$ is determined by a single twist parameter;  namely, how the circles are glued to each other. Thus the Strebel differentials  on any two rays  differ  by only a twist
about the core curve, and the amount of twisting is bounded by the length of the curve.   For any two points $r_1(t_1)$ and $r_2(t_2)$ along two such rays, if the moduli of the cylinders $M_1,M_2$ are equal and large, then $d(r(t_1),r_2(t_2))$ is small;
there is a  $O(1+1/M_1)$-quasiconformal map of the cylinders that realizes the twisting.  Standing Assumption II follows. 

Standing Assumption III holds since , as disscussed before Theorem \ref{theorem:uniqueness3}, the hypotheses of that theorem are satisfied by a product of Teichm\"{u}ller spaces.

\medskip

\noindent
{\bf Step 3 (\boldmath$\psi$ is continuous): }  Suppose  $x_n\in D_k(\M(S))$ converges to 
$z\in D_{k+m}(\M(S))$ as in Definitions \ref{df:top1} or \ref{df:top2}.  The proof of continuity of $\psi$ is by induction on $m$. Assume $m=1$.  If the  component of $D_k$ containing $x_n$ is a product, then by definition all of the  coordinates but one  of $\psi(x_n)$ in the product coincide with the corresponding coordinates of $x_n$. By assumption, these converge to the corresponding coordinates  of $\psi(z)$. Thus we can assume that the component of $D_k$ is not a (nontrivial) product.  Then  $\psi(z)$ is the Riemann surface $r(\infty)$, where $r$ is an EDM ray
in $D_k(\M(S))$, and $d_k(x_n,r(t_n))\to 0$ for a sequence $t_n\to\infty$.  The fact that $r(\infty)$ is the endpoint of $r$ says that $\psi(r(t_n))\to \psi(z)$ as $t_n\to\infty$ in the topology of 
$\DM$.   The fact that $d_k(x_n,r(t_n))\to 0$  says there is a sequence of $(1+o(1))$-quasiconformal maps  of 
 $\psi(x_n)$  to $\psi(r(t_n))$.  These converge to  a conformal map of a limit  $\psi(r(t_n))$ to $\psi(z)$.  Thus any such limit must in fact coincide with $\psi(z)$.

Now suppose the continuity of $\psi$ has been proved for all $p\leq m$ and  $m=p+1$. 
Again it suffices to assume that $D_k$ is not a product. Let $y_n$ a sequence in  $D_{k+1}(\M(S))$ such that $y_n\to z$ as in Definition \ref{df:top2}. There is a sequence of isolated rays $r_n$ in $D_k$ defined by one-cylinder Strebel differentials with core curve some $\gamma$ such that $y_n=r_n(\infty)$.      
By the induction hypothesis $\psi(y_n)\to \psi(z)$.  
Now assumption (2) in the definition of the topology implies that $$\text{Ext}_{r_n(t_n)}(\gamma)\to 0,$$  
for otherwise there would be rays in the same asymptote class whose distance from $r_n(t_n)$ does not tend  to $0$.  
Consider the $p+1$ nodes of $\psi(z)$ corresponding to pinching $p+1$ curves.  Without loss of generality we can assume the last $p$ of them are pinched along $\psi(y_n)$.   Form small neighborhoods of the corresponding paired punctures  on $\psi(z)$.  By definition of the topology, since     
 $\psi(y_n)\to \psi(z)$,  there is a conformal map of the complement of the last $p$ pair of neighborhoods to $\psi(y_n)$ for $n$ large. For each such  $n$, there is a conformal map of the complement of the first pair of neighborhoods to $r_n(t_n)$  for $t_n$  sufficiently large.  This shows that $\psi(r_n(t_n))\to \psi(z)$.  By assumption, there is a sequence of $(1+o(1))$-quasiconformal maps from $\psi(x_n)$ to $\psi(r_n(t_n))$, and therefore $\psi(x_n)\to \psi(z)$ as well. 
This shows that $\psi$ is continuous. 

\medskip
\noindent
{\bf Step 4 (Factoring $\psi$): }  Now the map $\psi$ itself is not injective, since one can have two combinatorially distinct $j$-tuples of curves which become combinatorially equivalent when one additional 
curve is added.  For example, if $S$ is closed of genus $2$, then in $D_2(\M(S))$ the component 
corresponding to pinching a separating and nonseparating curve is counted twice.  
However we show now that  the final identification Step 4 precisely identifies, by definition, such tuples.
Namely we show that the  map $\psi$  factors through a map 
$$\Psi:\bar{\M(S)}^{ir}\to \DM.$$

Suppose $z,z'\in D_{k+1}(\M(S))$ and $z\sim z'$.  We have to  show $\psi(z)=\psi(z')$.  By definition there  are
sequences $x_n,x_n'\in D_{k-1}(\M(S))$ that satisfy $d_{k-1}(x_n,x_n')\to 0$; $x_n\to z$, $x_n'\to z'$. 
By the continuity of $\psi$ we have $\psi(x_n)\to\psi(z)$ and $\psi(x_n')\to \psi(z')$. Since $
d_{k-1}(x_n,x_n')\to 0$, there is a sequence of $(1+o(1))$-quasiconformal maps from $\psi(x_n)$ to $\psi(x_n')$. Therefore we also have $\psi(x_n')\to \psi(z)$ and so $\psi(z)=\psi(z')$.  We have shown that there is a well-defined 
map $\Psi:\bar{\M(S)}^{ir}\to \DM$.

\medskip
\noindent
{\bf Step 5 ($\Psi$ is injective): }   We must prove that 
if $\Psi(z)=\Psi(z')$, then $z$ has been identified with $z'$.
We can assume $z,z'$ are in different components of $D_{k+1}(\M(S))$.  Again we can assume the components are not products; hence they are endpoints of rays $r,r'$ in  different components 
$E,E'$ of $D_k(\M(S))$.  Let $x_n\in D_{k-1}(\M(S))$ such that $x_n\to  z$. 
We wish to show $x_n\to z'$ as well, for then $z$ is identified with $z'$. 
We have  $X_n:=\Psi(x_n)\to Z:=\Psi(z)$.

\subsection { $(s,t)$ coordinate system}
Before continuing the proof we need to describe a coordinate system about $Z$  which allows us to represent any surface near $Z$  in the coordiante system. 
This coordinate system is due to \cite{EM} (see also \cite{Ma2} and \cite{W}). 
  
We may lift so that  $Z$ is in the augmented Teichmuller space. We will find a neighborhood $\V$ of $Z$ 
 whose  intersection   with $\Teich(S)$ will not be locally compact.  
 We can separate the nodes of $Z$ into pairs of punctures, denoted $p_i,q_i$.   
Choose conformal neighborhoods $V_i=\{z_i: 0<|z_i|<1\}$ and
$W_i=\{w_i:0<|w_i|<1\}$ of $p_i$ and $q_i$.  Also choose points $P$ and $Q$ on the pairs of circles of radius $1$. The discs  may be taken to be
mutually disjoint.  For each component $Z_l$ of $Z$ choose a nonempty  open
set $\W
$ disjoint from $\cup_i (V_i\cup W_i)$.  Let $n_l$ denote the complex  dimension of $\Teich(Z_l)$.  There exist Beltrami
differentials $\nu_1,\ldots, \nu_{n_l}$ supported  in $\W$ whose equivalence classes   form a basis for 
 the tangent space to $T_{Z_l}$ at $Z_l$. This implies  that for any $Y_l$ sufficiently close to $Z_l$, there is a $n_l$-tuple  $s(Y)=(s_1,\ldots s_{n_l})$ of complex numbers close to $0$ and a quasiconformal map $f:Z_l\to Y_l$ such that the dilatation  $\mu(f)$ of $f$ satisfies  
$$
\mu(f)=\sum_{i=1}^{n_l} s_i \nu_i.
$$

We do this for each component of $Z$.  The result is a parametrization of surfaces in a  neighborhood of $Z\in \V$  that lie in the bordification, by
$s\mapsto Z(s)$ for a neighborhood of $0$ in $\CC^N$, for some $N$.

Since the map $f$ (on each component) is conformal in $U_i\cup V_i$,  
the coordinates $z_i,w_i$ are local holomorphic coordinates in neighborhoods $V_i,W_i$ of the punctures on each $Z(s)$.   
Now choose a $p$-tuple $t=(t_1,\ldots, t_p)$  complex numbers in a small neighborhood of the  origin. 
 For each surface  $Z_s$, and for each $1\leq i\leq p$, remove the disc of radius $|t_i|^{1/2}$ from each of $V_i$ and $W_i$, and then  glue $z_i$ to $t_i/w_i$.  
We note that in this notation $Z(s,0)=Z(s)$; so if all $t_i=0$, then there are no disks to remove.

To define the neighborhod in $\Teich(S)$ we need to choose markings on $Z(s,t)$ 
by choosing a homotopy class of arcs joining $P$ and $Q$ crossing the  glued annulus. Thus we have a marking of the surface $Z(s,t)$ consisting of the marking of  $Z=Z(0,0)$, the curves along which we glued, and for each such curve, a transverse arc crossing the 
 annulus.  Note that markings differ by Dehn twists about the glued curve, and since these are arbitrary the resulting neighborhood is not locally compact.

 We continue the proof that $\psi$ is injective. We can lift to Teichmuller space and find the coordinate system $(s,t)$ around $Z$. 
Since $Z$ lies in a moduli space of two fewer dimensions than $X_n$, there are two plumbing coodinates $t_1,t_2$ such that
 the coordinates $t_1(n),t_2(n)$ of $X_n$ are both nonzero.   
 
We can assume  that points of $E'$ have  coordinate $t_1=0$, and  the $t_2$ coordinate tends to $0$ along the ray $r'(u)$ as $u\to\infty$. We can assume that points of $E$ have $t_2=0$.  The $s$ coordinate of $X_n$ approaches $0$.  For each $n$, we can find a time $u_n$ such that the modulus of the cylinder on $\Psi(r'(u_n))$ coincides with the modulus of the corresponding annulus on $X_n$.
For each such  $r'(u_n)$ there is a ray $r_n'\subset D_{k-1}(\M(S))$ such that $r'(u_n)=r_n'(\infty)$.  We can choose a time $l_n$ so that the corresponding cylinder on $r_n'(l_n)$ has the same modulus as the corresponding  annulus on $X_n$.
  Now, just as in the proof of Theorem~\ref{thm:distance},  as $n\to\infty$ there is a sequence of  
$(1+o(1))$-quasiconformal maps from $X_n$ to $\Psi(r'(l_n))$, and by the definition of the topology on the union of the $D_j(\M(S))$, we have that $x_n\to z'$. 

\medskip
\noindent
{\bf Step 6 ($\Psi^{-1}$ is continuous):  }Suppose then that $X_n\in \M(S')$ converges to $Z$ in $\DM$. 
Again we can form an $(s,t)$ coordinate neighborhood system  about $Z$
such that, after re-indexing, the $t$ coordinates of $X_n$ are given by  
$(t_1(n),\ldots, t_k(n))\neq 0$.  Here 
$k$ is the number of curves of $X_n$ that we pinch to get $Z$.
The proof is by induction on $k$ and resembles the proof that $\Psi $ is injective.  
  Suppose $k=1$. Let $r$ be the Strebel ray
with endpoint $r(\infty)=Z$, so by definition, $\Psi([r])=Z$.
For each $n$, we can find a time $u_n$ such that the modulus of
the cylinder on $r(u_n)$ is the same as the modulus about the
pinching curve on $X_n$ found by the plumbing
construction.  Now
again just as in the proof of Theorem~\ref{thm:distance}, for any
$\epsilon$, for $n$ large enough, we can find a
$(1+\epsilon)$-quasiconformal map from $X_n$ to $r(u_n)$.  
Then
by definition, $X_n\to [r]=\Psi^{-1}(Z)$ in the topology of
$\bar{\M(S)}^{ir}$.

Now for the induction step.  Suppose we have proven the desired
limit for $k-1$, where $Z$ is found by pinching along $k$ curves. We
have $Z=\Psi([r_0])$ for some ray $r_0$.  Let $Y_n$ have the same $(s,t)$ 
coordinates as $X_n$ except that we require $t_1=0$. This means
that we find $Z$ from $Y_n$ by pinching $k-1$ curves.  Let $q_n$
be the Strebel differential on $Y_n$ with double poles at the punctures
corresponding to $t_1=0$, and let $r_n$ be the corresponding
Strebel ray with endpoint $r_n(\infty)=Y_n$.  By definition,
$\Psi([r_n])=Y_n$.  Now $Y_n\to Z$ in $\DM$, and by the induction
hypothesis on the continuity of the map $\Psi^{-1}$, we see that
$[r_n]\to [r_0]$. Just as above we may choose $u_n$ so that the
modulus of the cylinder on $r_n(u_n)$ is the same as the
modulus of the annulus corresponding to the $t_1$ coordinate in 
the plumbing
construction.  By definition of the topology of
$\bar{\M(S)}^{ir}$ it is again enough to prove that $d_{\M(S)}(X_n,
r_n(t_n))\to 0$.  But this again follows just as in the proof of
Theorem~\ref{thm:distance}:  there is a conformal map 
$X_n\to \Psi(r(u_n))$ in the complement of annuli with large but equal
moduli; then for any $\epsilon$, for $n$ large enough, we can
find a $(1+\epsilon)$-quasiconformal map from $X_n$ to $\Psi(r(u_n))$.
This completes the proof.
\end{proof}

\section{Further geometric properties}

\subsection{A strange example}

In this subsection we indicate some of the difficulties of the Teichm\"{u}ller geometry of $\M(S)$ by exhibiting two sequences of EDM rays $r_n,r'_n$, with the following properties:  there exists a constant $D>0$ and 
sequences of times $t_n, t'_n\to\infty$ such that $d_{\M(S)}(r_n(t_n),r'_n(t'_n))\leq D$, each 
sequence $r_n, r_n'$ converges to an EDM ray $r_\infty, r'_\infty$ uniformly on compact intervals of time, and yet $r_\infty$ does not stay within a bounded distance of $r'_\infty$.  
This example violates Assumption 9.11 of  \cite{JM}, so that 
the Ji-MacPherson compactification method cannot be applied to $\M(S)$.  This partially explains why we took a different approach.

We construct a sequence of rays $r_n$ as follows.  Let
$r_0$ be a Strebel ray corresponding to a maximal collection of
curves $\beta_1,\ldots, \beta_{3g-3+n}$ whose cylinders have
equal moduli. Note that $r_0(\infty)$ is the unique maximally
noded Riemann surface within its combinatorial equivalence class.
Let $\alpha$ be a curve distinct from the $\beta_i$ and therefore
it has positive intersection with some $\beta_j$.  Let $T_\alpha$ denote the Dehn twist about $\alpha$.  Let $r_n$ be
the Strebel ray through $r_0(0)$ corresponding to the Strebel
differential whose set of core curves is $\{T_\alpha^n(\beta_i)\}$ and whose cylinders have equal
moduli. This is possible by a theorem of Strebel (\cite{St}, Theorem 21.7).  Note
that $r_n(\infty)=r_0(\infty)$ for each $n$, since the
collection $\{T_\alpha^n(\beta_i)\}$ is combinatorially
equivalent to $\{\beta_i\}$.   Since the rays are modularly
equivalent they are asymptotic (Corollary  \ref{cor:asym} above), so we can choose
times $t_n, t_n'\to\infty$ such that $d_{\M(S)}(r_n(t_n),r_0(t_n'))$ is
uniformly bounded.

On the other hand the rays $r_n$ converge uniformly on compact
sets in time to a ray $r_\infty$, where $r_\infty$ corresponds 
to the unique one cylinder Strebel
differential with core curve $\alpha$.  Taking $r_n'=r_0$ so that
$r_\infty'=r_0$ for all $n$, we have $d(r_\infty,r_\infty')=\infty$ by
Lemma~\ref{lem:short1}. 

\subsection{The set of asymptote classes of all EDM rays}

In this subsection we give a parametrization of the set of asymptote classes of all (not necessarily isolated) EDM rays.  As we will see, this space is naturally a closed simplex bundle $B$ over 
$\DM$.  Let $S$ be a surface of genus $g$ with $n$ punctures.  
The fiber over a point $\hat X\in\M_{g',n'}$, where $(g',n')\neq (g,n)$,  
 consists of projective classes $(b_1,\ldots, b_p)$ of vectors.  Let $\Sigma$ be the collection of all 
asymptotic classes of EDM rays  on $\M_{g,n}$. 
We define a map $$\Phi:\Sigma\to \B.$$  

Let $[r]$ be an equivalence class of rays.  Let $r$ any
representative with cylinders $C_1,\ldots, C_p$ with moduli
$\mod(C_1),\ldots, \mod(C_p))$.  By
Corollary~\ref{cor:asym} the projective class of the vector of
moduli is independent of the choice of representative and the
endpoint $r(\infty)$ is independent of the representative. Define
$\Phi([r])$ to be the point whose base is $r(\infty)$ and whose
fiber is the projective vector $(\mod(C_1),\ldots, \mod(C_p))$

\begin{theorem}
The map $\Phi$ is a homeomorphism onto the open simplex subbundle $\B_0$ where no coordinate is $0$.
\end{theorem}
\begin{proof}
The map $\Phi$ is clearly injective. To show surjectivity let  $\hat X\in \M_{g',n'}$ any point;    $v=(M_1,\ldots M_j)$ a projective vector.  Pick a representative vector $v$ and let $(\hat X,\hat q)$ be the (unique) quadratic differential on $\hat X$ such that 
\begin{itemize}
\item $(\hat X,\hat q)$ has   double poles at the punctures,
\item  the vertical trajectories are closed loops isotopic to the punctures
\item   the lengths of the vertical trajectories are $1/M_i$ for each paired puncture.
\end{itemize}
This is possible by Theorem 23.5 of \cite{St}.  
 Remove a punctured disc around each paired puncture so that the remaining cylinder has height $1/2$.  Glue together along the circles.  The corresponding cylinders $C_i$ have height $1$.  The moduli of the cylinders are therefore $M_i$. We may choose the representative $v$ so that the area of the resulting $(X,q)$ is $1$. 
This gives a corresponding geodesic ray $r(t)$.  We have that $\hat X=r(\infty)$, 
so that  $\Phi([r])=(\hat X, M_1,\ldots, M_j)$  
 
The quadratic differential $(\hat X,\hat q)$ depends continously on $\hat X$ and
the vector $v$, which implies that the ray $[r]$ depends continuously on these parameters so that the map $\Phi^{-1}$ is continuous. The map $\Phi$ is continuous because the endpoints and moduli depend continously on the quadratic differentials defining the ray. 
\end{proof}

\subsection{Tits geometry of the space of EDM rays}
\label{section:tits}

In this section we compute some invariants for pairs of EDM rays.  These invariants are fundamental in the study of nonpositively curved manifolds (see, e.g., \cite{Eb}, Chapter 3).  

\begin{definition}
Let $r(t),r'(t)$ a pair of EDM rays in a metric space $(X,d)$.  We define the {\em pre-Tits distance} 
$\ell(r,r')$ between $r$ and $r'$ to be 
$$\ell(r,r'):=\lim_{t\to\infty}
\frac{d(r(t),r'(t))}{t}$$ if the limit exists.  
\end{definition}

For simply-connected, nonpositively curved manifolds $X$, the {\em Tits distance} on the visual boundary $\partial X$ is equal to the path metric induced by $\ell$ (\cite{Eb}, Prop. 3.4.2).  The quantity $\ell$ is related to the {\em angle metric} $\angle(r,r')$ on $\partial X$ via 
$$\ell(r,r')=2\sin(\frac{1}{2}\angle(r,r'))$$
(see \cite{Eb}, Prop. 3.2.2).

Our goal now is to compute $\ell$ for pairs of EDM rays in $\M(S)$.

\begin{theorem}
Let $r,r'$ be EDM rays defined by Strebel differentials
$(X,q)$ and $(X',q')$ with core curves $\{\gamma_i\}$ 
and $\{\gamma_j'\}$. 
The Tits angle between $r$ and $r'$ is $0$ if  there is an element $\phi$ of the mapping class
group sending $\{\gamma_i\}_{i=1}^p$ to $\{\gamma_j'\}_{i=i}^{p'}$.  The angle is $1$ if
the above does not hold but there is an element $\phi$ of the mapping
class group such that $i(\phi(\gamma_i),\gamma_j')=0$ for all
$\gamma_i,\gamma_j'$.  The angle is $2$ otherwise.
\end{theorem}

This discretization of Tits angles lies in contrast to what happens for higher rank locally symmetric spaces $\Gamma\backslash G/K$, where one has 
a continuous values of the Tits angles coming from almost isometrically embedded Weyl chambers. 
 
\begin{proof}
The first case is if the collection of curves $\{\gamma_i\}$ is combinatorially equivalent to the collection of curves $\{\gamma_j'\}$.
That is, there is an element $\phi$ of the mapping class group sending one collection to the other. Then the corresponding geodesics stay bounded distance apart by \cite{Ma1}.  Thus the Tits angle is $0$.

Thus assume the collections are not combinatorially equivalent. Assume further that any collection of curves combinatorially equivalent to $\{\gamma_i\}$ must intersect some $\gamma_j'$. By reindexing we can assume  $$i(\gamma_1,\gamma'_1)>0.$$  Now by Lemma~\ref{lem:short1}
$$e^{2t}\Ext_{r(t)}(\gamma_1)\to c_1,$$ for some
$c_1>0$. Since $\gamma_1$ crosses $C_1'$,
$$\Ext_{r'(t)}(\gamma_1)\geq c_2e^{2t},$$ for some $c_2>0$.
 By Theorem~\ref{theorem:kerckhoff:formula} $$d_{\M(S)}(r(t),r'(t))\geq 1/2\log
(c_1c_2e^{4t})$$ and so $$\liminf_{t\to\infty}\frac
{d_{\M(S)}(r(t),r'(t))}{t}\geq 2.$$ On the other hand by the triangle inequality
$$\limsup_{t\to\infty}\frac {d_{\M(S)}(r(t),r'(t))}{t}\leq 2,$$ and we are done
in this case. 

The remaining case is that there is some $\phi$ so that $i(\phi(\gamma_i),\gamma_j')=0$ for all $i,j$. 
 There are several possibilities  with similar analyses.    Assume for example  that after reindexing and 
 applying an element of $\Mod(S)$ that $\gamma_1\neq \gamma_j'$ for all $j$.  
Now since  $$i(\gamma_1,\gamma_j')=0$$ for all $j'$,   by Lemma \ref{lem:short1}  we have $\text{Ext}_{r'(t)}(\gamma_1)$ bounded below,  and so by Theorem \ref{theorem:kerckhoff:formula} 
$$\liminf_{t\to\infty} \frac {d_{\M(S)}(r(t),r'(t))}{t}\geq 1.$$
We need to show the opposite inequality. 
That is, we need to show  
\begin{equation}
\label{eq:bounded2}
\sup_{\beta} \frac{Ext_{r(t)}(\beta)}{Ext_{r'(t)}(\beta)}\leq c(t)e^{2t},
\end{equation} where 
$$\frac{\log c(t)}{t}\to 0.$$
 
We will use results of Minsky \cite{Mi1} to compare extremal lengths of any $\beta$ along $r(t)$ and
$r'(t)$.   We will say that two functions $f,g$ are {\em comparable}, denoted $f\asymp g$, if 
$f$ and $g$ differ by fixed multiplicative constants (which in our case will depend only on the genus of $S$).  

Fix some $\epsilon>0$, smaller than the Margulis constant for $S$.  
For sufficiently large $t_0$, and for each cylinder $C_i$ along $r(t)$,  find a pair of curves $\gamma_i^1,\gamma_i^2$ with the following properties:
\begin{enumerate}
\item  $\gamma_i^1,\gamma_i^2$ are isotopic to  $\gamma_i$.
\item  Each has fixed  hyperbolic length $\epsilon$. 
\item $\gamma_i^1$ and $\gamma_i^2$  bound a cylinder $\hat C_i\subset C_i$
such that $\frac{\text{mod}(\hat C_i)}{\text{mod}(C_i)}\to 1$ as $t\to\infty$. 
\end{enumerate}
Note that 
$$\text{mod}(C_i)=c_ie^{2t}$$ 
for some fixed $c_i$. 
Let $M_i(t)=\text{mod}(\hat C_i)$.    The curves $\gamma_i^j; j=1,2$ define the thick-thin decomposition of
$r(t)$.  The components $\Omega_j$ of the complement of the cylinders $\hat C_i$ are  thick. 
According to \cite{Mi1}, for any $\beta$ we have 
\begin{equation}
\label{eq:last}
\text{Ext}_{r(t)}(\beta)\asymp \max_{i,j} (\text{Ext}_{\hat C_i}(\beta),\text{Ext}_{\Omega_j}(\beta)),
\end{equation} 
which is the maximum of the contribution to the extremal length of $\beta$ from its intersections with the $\hat C_i$ and  the $\Omega_j$.  These  quantities are given below. 

For the first, 
the hyperbolic  geodesic representative of $\beta$ crosses each $\hat C_i$ a total of $n_i$ times, twisting $t_i$ times. The contribution to extremal length $\text{Ext}_{\hat C_i}(\beta)$ from its intersection with  $\hat C_i$ is given by 
\begin{equation}
\label{eq:cylinder}
\text{Ext}_{\hat C_i}(\beta)=n_i^2(M_i(t)+t_i^2/M_i(t)).
\end{equation}
By \cite{Mi1} 
the contribution to extremal length $\text{Ext}_{\Omega_j}(\beta)$ of $\beta$ from  $\Omega_j$ is comparable to 
$\ell^2(\beta\cap\Omega_j)$, where $\ell(\cdot)$ is length in the hyperbolic metric.   
This quantity can be computed as follows.   
Let $\Gamma_j=\Gamma\cap \Omega_j$, the component  of the critical graph contained in $\Omega_j$.  Choose generators $\omega_1,\ldots,\omega_n$ for $\pi_1(\Gamma_j)$, where $n=n(j)$.   
Since $\Omega_j$ is thick, we have
\begin{equation}
\ell^2(\beta\cap\Omega_j)\asymp(\max_i i(\beta,\omega_i))^2.
\end{equation}
and so 
\begin{equation}\label{eq:thick}
\Ext_{\Omega_j}(\beta)\asymp (\max_i i(\beta,\omega_i))^2
\end{equation}
Similar estimates hold for the extremal length of $\beta$ on $r'(t)$.
Now assume   $\beta$ crosses   $C_1$.  
By assumption,  the core curve  $\gamma_1$ of $C_1$ lies in 
a thick component $\Omega_j'$ of $r'(t)$.  By  (\ref{eq:thick}), the contribution  to the extremal length of $\beta$ in the thick part of $\Omega_j'$ from the $n_i$ crossings of $\beta$ with $\gamma_1$ with $t_i$ twists, is comparable to $n_i^2t_i^2$.  The contribution to extremal length of intersections with curves whose homotopy classes lie in both critical graphs are comparable, by (\ref{eq:thick}). Comparing the estimate $n_i^2t_i^2$ to (\ref{eq:cylinder}) we see that for some $c>0$,
$$\frac{\Ext_{r(t)}(\beta)}{\Ext_{r'(t)}(\beta)}\leq c\frac{n_i^2(M_i(t)+t_i^2/M_i(t)}{n_i^2t_i^2}\leq cM_i(t)\leq cc_ie^{2t}.$$
The same estimates hold if $\beta$ crosses a collection of $\hat C_i'$ while the $\gamma_i'$ lie in  thick components  $\Omega_j$.   Thus we see that (\ref{eq:bounded2}) holds.
\end{proof}

\section{Appendix: Proof of Theorem \ref{theorem:uniqueness3}}
\label{section:appendix}

Before we begin the proof of Theorem \ref{theorem:uniqueness3} we will need some definitions and lemmas.   By a {\em geodesic} in a metric space we will mean a globally length-minimizing geodesic.  Suppose $Y=Y_1\times
\ldots \times Y_m$ is a product of metric spaces, given the sup metric.  A pair of points $p=(p_1,\ldots, p_m)$ and $q=(q_1,\ldots , q_m)$ in $Y$ 
is called a {\em diagonal pair} if  
$d_{Y_i}(p_i,q_i)=d_{Y_j}(p_j,q_j)$ for  $1\leq i,j\leq m$.  If one of the points is understood, we call the other a {\em diagonal point}.

The following lemma follows directly from the definition of the 
$\sup$ metric on $Y$.

\begin{lemma}[Characterizing diagonal pairs]
\label{lemma:ap2}
Let $Y$ be as above, and suppose $m\geq 2$.  If $p,q$ is a diagonal pair, then any 
geodesic between $p$ and $q$ is of the form $(r_1(t),\ldots, r_m(t))$, where 
each $r_i(t)$ is a geodesic segment in $Y_i$, and the $r_i(t)$ have the same parametrizations.    
Thus if there is a unique geodesic from $p_i$ to $q_i$ for each $1\leq i\leq m$, then 
there is a unique geodesic from $p$ to $q$.  If $p,q$ is not a diagonal pair, then there are 
infinitely many geodesics in $Y$ from $p$ to $q$.
\end{lemma}

Now suppose that we are in the situation of the hypotheses of Theorem \ref{theorem:uniqueness3}.  By the previous paragraph, Lemma \ref{lemma:ap2}, and the definition of the sets $S_i$, we have the following.

\begin{lemma}
\label{lemma:ap3}
The set of points  
$$(S_1\times Y_2\cdots \times Y_m)\cup (Y_1\times S_2\times\cdots\times Y_m) \cup 
(Y_1\times\cdots\times Y_{m-1}\times S_m)$$
in $Y$ is precisely the set of points
$z=(z_1,\ldots ,z_m)\in Y$ with the following property: 
there exists an integer $N>1$ such that for every neighborhood $U$ of $z$, there exists 
a pair of points $x,y\in U$ such that the number of geodesics in $Y$ 
from $x$ to $y$ is greater than one and at most $N$. In fact we can take $N=N_1\cdot N_2\cdots N_m$.
\end{lemma}

Note that the complement of the  set given in Lemma \ref{lemma:ap3} is just $(Y_1\setminus S_1)\times \cdots \times (Y_m\setminus S_m)$.  The characterization of the points in the set given by Lemma \ref{lemma:ap3} is purely metric, and is therefore clearly preserved by any isometry of $Y$ and therefore so is its complement. 
It follows that any isometry of $Y$ preserves this set.  But the metric space $(Y_1\setminus S_1)\times \cdots \times (Y_m\setminus S_m)$ is a product of geodesic metric spaces, none of which is a point, and each of which has the locally unique geodesics property.  Malone \cite{Mal} proved that any such product decomposition (in the $\sup$ metric) is unique.  As each $(Y_i\setminus S_i)$ is open and dense in $Y_i$, any isometry of $Y_i\setminus S_i$ has a unique extension to $Y_i$. Theorem \ref{theorem:uniqueness3} follows.

\noindent
Benson Farb:\\
Dept. of Mathematics, University of Chicago\\
5734 University Ave.\\
Chicago, Il 60637\\
E-mail: farb@math.uchicago.edu
\medskip

\noindent
Howard Masur:\\
Dept. of Mathematics, University of  Chicago\\
5734 University Ave\\
Chicago, IL 60637\\
E-mail: masur@math.uchicago.edu


\begin{thebibliography}{ABCDEF}
\small


\bibitem[Be]{Be}
L. Bers,  Spaces of degenerating Riemann surfaces, in {\em Discontinuous groups and Riemann surfaces} (Proc. Conf., Univ. Maryland, College Park, Md., 1973), pp. 43--55. Ann. of Math. Studies, No. 79, Princeton Univ. Press, Princeton, N.J., 1974, MR0361051

\bibitem[DM]{DM}
P. Deligne and D. Mumford, The irreducibility of the space of curves of given genus, {\em Publ. Math. IHES}, No. 36 (1969), 75--109, MR0262240, Zbl 0181.48803.

\bibitem[Eb]{Eb}
P. Eberlein, {\em Geometry of nonpositively curved manifolds}, Chicago Lectures in Math., 
1996, MR1441541, Zbl 0883.53003.

\bibitem[EK]{EK}
C. Earle and I. Kra, On isometries between Teichm?ller spaces, {\em Duke Math. J.} 41 (1974), 583--591, MR0348098, Zbl 0293.32020.

\bibitem[EM]{EM}
C. Earle and A. Marden, in preparation.

\bibitem[Hu]{Hu}
J. Hubbard, {\em Teichm\"{u}ller Theory}, Vol. 1, Matrix Editions, 2006, MR2245223, 
Zbl 1102.30001.

\bibitem[HS]{HS}
P.Hubert, T.Schmidt, An introduction to Veech surfaces, Handbook of Dynamical Systems, Vol. 1B Katok and Hasselblatt eds. Elsevier, 200, MR2186246, 
Zbl 1130.37367. 

\bibitem[JM]{JM}
L. Ji and R. MacPherson, Geometry of compactifications of locally
symmetric spaces, {\em Ann. Inst. Fourier, Grenoble}, Vol. 52, No. 2
(2002), 457--559, MR1906482, Zbl 1017.53039.

\bibitem[Ke]{Ke}
S. Kerckhoff, The asymptotic geometry of Teichmuller space, {\em
Topology}, Vol. 19, 23--41, MR0559474, Zbl 0439.30012 

\bibitem[Mal]{Mal} 
W.Malone, Isometries of Products of Path-Connected Locally Uniquely Geodesic Metric Spaces with the Sup Metric are Reducible, preprint, February 2010.

\bibitem[Ma1]{Ma1}
H. Masur, On a class of geodesics in Teichm\"{u}ller space, {\em Annals
of Math.}, Vol. 102, No. 2 (Sep. 1975), 205--221, MR0385173, Zbl 0322.32010 .

\bibitem[Ma2] {Ma2} 
H. Masur, {\em Extension of the Weil-Petersson metric to the boundary of \hbox{Teich}m\"uller space}, Duke Math J. 43 (1976) 623-635, MR0417456, Zbl 0358.32017. 

\bibitem[Mas]{Mas}
B. Maskit, 
Comparison of hyperbolic and extremal lengths, {\em 
Ann. Acad. Sci. Fenn} 10 (1985) 381--386, MR0802500, Zbl 0587.30043.

\bibitem[Mc]{Mc}
C. McMullen, The moduli space of Riemann surfaces is Kahler hyperbolic,
{\em Annals of Math.} (2), Vol. 151 (2000), no. 1, 327--357, MR1745010, Zbl 0988.32012.

\bibitem[Mi1]{Mi1}
Y. Minsky, Extremal length estimates and product regions in 
Teichm\"{u}ller space, {\em Duke Math. Jour.}  83  (1996),  no. 2, 
249--286, MR1390649, Zbl 0861.32015.

\bibitem[Mi2]{Mi2}
Y. Minsky, Harmonic maps, length, and energy in Teichm\"{u}ller space, 
{\em J. Differential Geom.} 35  (1992),  no. 1, 151--217., MR1152229, Zbl 0763.53042.

\bibitem[MS]{MS}
H. Masur and J. Smillie,  Hausdorff dimension of sets of nonergodic measured foliations, {\em Annals 
of Math.}, 2nd Ser., Vol. 134, No. 3 (1991), 455--543, MR1135877, Zbl 0774.58024.

\bibitem[St]{St}
K.Strebel, {\em Quadratic differentials}, Ergebnisse der Math. (3), Vol. 5. Springer-Verlag, Berlin, 1984, MR743423, Zbl 0547.30001.


\bibitem [W]{W} Wolpert, Scott A. {\em Geometry of the Weil-Petersson completion of Teichm\"uller space.}  Surveys in differential geometry, Vol. VIII (Boston, MA, 2002),  357--393, Surv. Differ. Geom., VIII, Int. Press, Somerville, MA, 2003. 


\end{thebibliography}
\end{document}